\documentclass[a4paper]{article}

\usepackage[english]{babel}
\usepackage[utf8x]{inputenc}
\usepackage[T1]{fontenc}
\usepackage{graphicx}
\usepackage{caption, subcaption}
\usepackage[title]{appendix}
\usepackage{tikz}
\usepackage{pgfplots}
\usepackage{cutwin}
\graphicspath{ {images/} }
\usetikzlibrary{intersections,decorations.pathreplacing,decorations.markings,calc,angles,quotes,arrows.meta,pgfplots.fillbetween,patterns}
\usepackage[a4paper,top=2.8cm,bottom=3cm,left=2.7cm,right=2.7cm,marginparwidth=1.75cm]{geometry}
\usepackage{amsmath,yhmath,amsfonts}
\usepackage{amsthm}
\usepackage{amssymb}
\usepackage{hyperref}
\usepackage{cleveref}
\usepackage{epstopdf}
\usepackage{comment}
\usepackage[]{todonotes} %[disable]
\usepackage{color}
\usepackage{float}
\usepackage[sort&compress,numbers]{natbib}
\newtheorem{thm}{Theorem}% 	 	[section]
\newtheorem{lem}[thm]{Lemma}%[section]
\newtheorem{prop}[thm]{Proposition}%[section]

\theoremstyle{definition}
\newtheorem{rmk}[thm]{Remark}%[section]
\numberwithin{equation}{section}
\numberwithin{thm}{section}

\newcommand{\A}{\mathcal{A}}

\newcommand{\C}{\mathcal{C}}
\newcommand{\Cb}{\mathbb{C}}

\newcommand{\F}{\mathcal{F}}

\newcommand{\K}{\mathcal{K}}

\renewcommand{\P}{\mathcal{P}}
\newcommand{\R}{\mathbb{R}}

\renewcommand{\S}{\mathcal{S}}
\newcommand{\T}{\mathcal{T}}
\newcommand{\U}{\mathcal{U}}

\newcommand{\Z}{\mathbb{Z}}
\newcommand{\p}{\partial}
\renewcommand{\epsilon}{\varepsilon}
\newcommand{\dx}{\: \mathrm{d}}

\renewcommand{\u}{\mathbf{u}}
\newcommand{\uf}{\mathfrak{u}}
\renewcommand{\v}{\mathbf{v}}
\newcommand{\w}{\mathbf{w}}
\renewcommand{\Re}{\mathrm{Re}}

\newcommand{\ie}{\textit{i.e.}}
\newcommand{\nm}{\noalign{\smallskip}}
\newcommand{\ds}{\displaystyle}
\newcommand{\iu}{\mathrm{i}\mkern1mu}

\newcommand{\svdots}{\raisebox{3pt}{\scalebox{.6}{$\vdots$}}}
\newcommand{\sddots}{\raisebox{3pt}{\scalebox{.6}{$\ddots$}}}
\newcommand{\sadots}{\raisebox{3pt}{\scalebox{.6}{$\adots$}}}

\makeatletter
\newcommand{\neutralize}[1]{\expandafter\let\csname c@#1\endcsname\count@}
\makeatother

\title{Edge modes in active systems of subwavelength resonators}
\author{
	Habib Ammari\thanks{\footnotesize Department of Mathematics, 
		ETH Z\"urich, 
		R\"amistrasse 101, CH-8092 Z\"urich, Switzerland (habib.ammari@math.ethz.ch, erik.orvehed.hiltunen@sam.math.ethz.ch).}\and Erik Orvehed Hiltunen\footnotemark[1] }
\date{}

\begin{document}
	\maketitle
	
	\begin{abstract}
		Wave scattering structures with amplification and dissipation can be modelled by non-Hermitian systems, opening new ways to control waves at small length scales. In this work, we study the phenomenon of topologically protected edge states in acoustic systems with gain and loss. We demonstrate that localized edge modes appear in a periodic structure of subwavelength resonators with a defect in the gain/loss distribution, and explicitly compute the corresponding frequency and decay length. Similarly to the Hermitian case, these edge modes can be attributed to the winding of the eigenmodes. In the non-Hermitian case, the topological invariants fail to be quantized, but can nevertheless predict the existence of localized edge modes.
	\end{abstract}
\vspace{0.5cm}
	\noindent{\textbf{Mathematics Subject Classification (MSC2000):} 35J05, 35C20, 35P20.
		
\vspace{0.2cm}

\vspace{0.5cm}	
%	\def\keywords2{\vspace{.5em}{\textbf{  Mathematics Subject Classification
%				(MSC2000).}~\,\relax}}
%	\def\endkeywords2{\par}
%	\keywords2{35R30, 35C20.}
%	
%	\def\keywords{\vspace{.5em}{\textbf{ Keywords.}~\,\relax}}
%	\def\endkeywords{\par}
%	\keywords{subwavelength resonance, high contrast, subwavelength phononic crystal, topological insulator, edge mode.}
	
	\section{Introduction}	
	In classical wave systems, sources of amplification and dissipation can be modelled by non-real material parameters. Consequently, the underlying system is non-Hermitian, meaning that the left and right eigenmodes are distinct. This opens the possibility of \emph{exceptional points}, which are parameter values such that the left and right eigenmodes are orthogonal, or equivalently, such that the system is not diagonalizable. Around such points, rich physical phenomena have been observed, including enhanced sensing and unidirectional invisibility (see, for example, \cite{heiss2012physics,miri2019exceptional} for overviews of physical properties around exceptional points). A particular case of non-Hermitian systems are systems with \emph{parity-time symmetry}, or \emph{$\P\T$ symmetry}. The spectrum of a $\P\T$-symmetric system is conjugate-symmetric, and is therefore either real (known as \emph{unbroken} $\P\T$ symmetry) or non-real and symmetric around the real axis (known as \emph{broken} $\P\T$ symmetry).
	
	In the study of topologically protected edge modes, eigenvalue degeneracies are lifted to open band gaps. In the well-known \emph{Su-Schrieffer-Heeger model} \cite{SSH}, a certain parameter choice corresponds to a conical degeneracy known as a \emph{Dirac cone}. As the parameter varies, the degeneracy can open into two topologically distinct band gaps. In the Hermitian case, the topological properties of one-dimensional insulators can be described by the \emph{Zak phase}. This is a geometrical phase which describes the winding of the eigenmodes as the wave vector is varying. The \emph{bulk-boundary correspondence} states that, by combining materials with different Zak phases, the total structure will support modes that are confined to the interface between the two materials. These modes as known as \emph{edge modes} \cite{drouot1, fefferman,fefferman_mms}.
	
	In the non-Hermitian case, the exceptional point degeneracies can open into non-trivial band gaps enabling topologically protected non-Hermitian edge modes. Initially, such modes were created by adding gain and loss to structures which already in the Hermitian case support edge modes, which enables selective enhancement of the edge modes \cite{schomerus2013topologically,poli2015selective,weimann2017topologically,parto2018edge}. Later, it was discovered that pure non-Hermitian edge modes can be created, \ie{}, edge modes that originate purely from the gain/loss distribution and cease to exist in the Hermitian limit \cite{takata2018photonic,leykam2017edge}. Moreover, non-Hermitian effects can be introduced by having anisotropic couplings. This can give rise to  the \emph{skin effect}, where bulk modes are localized to the edges of the structure, and has been used for efficient funnelling of waves \cite{weidemann2020topological}. The two different origins of the edge modes have been explained by two different topological winding numbers: the (non-Hermitian) Zak phase and the \emph{vorticity}, where the latter describes the winding of the complex eigenvalues \cite{midya2018non,ghatak2019new,leykam2017edge,ni2018pt,yuce2018pt,lang2018effects,okuma2020topological,yao2018edge}.
	
	Protected edge modes in non-Hermitian systems have typically been studied for tight-binding Hamiltonians with \emph{chiral} symmetry. This enables a natural generalization of the Zak phase which is quantized \cite{yin2018geometrical}. In the case without chiral symmetry, a similar approach can be made, but results in a continuously varying Zak phase \cite{jiang2018topological,lieu2018topological}. Instead, the ``total Zak phase'' can be considered, \ie{}, the sum of the Zak phases for each band, which is indeed quantized \cite{jiang2018topological,liang2013topological}. In this work, however, we demonstrate the existence of edge modes even in structures whose total Zak phase vanishes. Instead, these modes can be attributed to a non-zero (individual) Zak phase, and the continuously varying Zak phase can be interpreted as a ``partial'' band inversion. By combining two materials whose Zak phases have opposite sign, edge modes appear along the interface.
	
	Non-Hermitian material parameters provide a way to have protected edge modes in crystals where the periodic geometry is intact, and a defect is placed in the parameters. Due to the quantized Zak phase, this is not possible in the non-Hermitian case \cite{ammari2019topologically}. In this work, we study an array consisting of dimers of subwavelength resonators, with periodic geometry and a general configuration of the bulk modulus. We begin by studying the periodic case, and compute the vorticity and the Zak phase. We then introduce a defect in the bulk modulus, and explicitly compute the frequency of localized modes in the subwavelength regime. The modes created this way originate purely from the non-Hermitian gain and loss. In addition, we demonstrate numerically the edge modes in a non-Hermitian analogue of the system studied in \cite{ammari2019topologically}, where edge modes exist even in the Hermitian case.
	
	\section{Problem statement and preliminaries} \label{sec:prelim}
	In this section, we define the structure under consideration, and discuss some preliminary theory needed for the analysis.
		
	\subsection{Problem statement}
	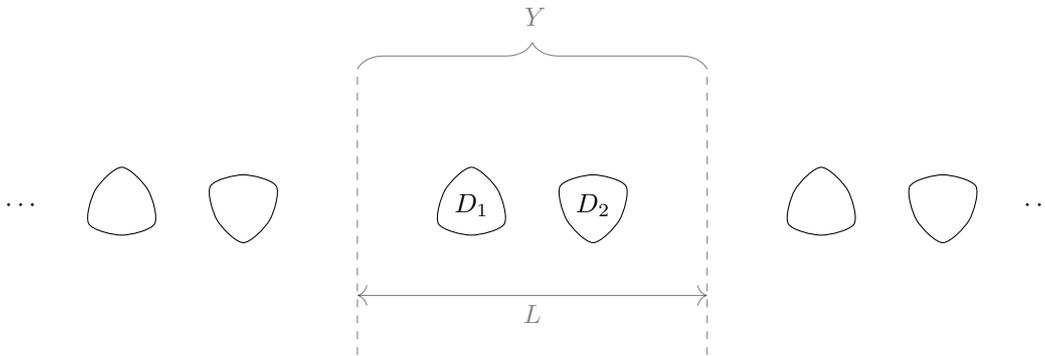
\begin{figure}[tbh]
		\centering
		\begin{tikzpicture}[scale=2]
		\pgfmathsetmacro{\rb}{0.25pt}
		\pgfmathsetmacro{\rs}{0.2pt}
		\coordinate (a) at (0.25,0);		
		\coordinate (b) at (1.05,0);	
		
		\draw[dashed, opacity=0.5] (-0.5,0.85) -- (-0.5,-1);
		\draw[dashed, opacity=0.5]  (1.8,0.85) -- (1.8,-1)node[yshift=4pt,xshift=-7pt]{};
		\draw[{<[scale=1.5]}-{>[scale=1.5]}, opacity=0.5] (-0.5,-0.6) -- (1.8,-0.6)  node[pos=0.5, yshift=-7pt,]{$L$};
		\draw plot [smooth cycle] coordinates {($(a)+(210:\rb)$) ($(a)+(270:\rs)$) ($(a)+(330:\rb)$) ($(a)+(30:\rs)$) ($(a)+(90:\rb)$) ($(a)+(150:\rs)$) }; \draw (a) node{$D_1$};
		\draw plot [smooth cycle] coordinates {($(b)+(30:\rb)$) ($(b)+(90:\rs)$) ($(b)+(150:\rb)$) ($(b)+(210:\rs)$) ($(b)+(270:\rb)$) ($(b)+(330:\rs)$) }; \draw (b) node{$D_2$};
		
		%\draw[{<[scale=1.5]}-{>[scale=1.5]}, opacity=0.5] (0.25,0.6) -- (1.05,0.6) node[pos=0.5, yshift=-5pt,]{$l$};
		%\draw[dotted,opacity=0.5] (0.25,0.7) -- (0.25,-0.8) node[at end, yshift=-0.2cm]{$p_1$};
		%\draw[dotted,opacity=0.5] (1.05,0.7) -- (1.05,-0.8) node[at end, yshift=-0.2cm]{$p_2$};
		
		\begin{scope}[xshift=-2.3cm]
		\coordinate (a) at (0.25,0);		
		\coordinate (b) at (1.05,0);	
		\draw plot [smooth cycle] coordinates {($(a)+(210:\rb)$) ($(a)+(270:\rs)$) ($(a)+(330:\rb)$) ($(a)+(30:\rs)$) ($(a)+(90:\rb)$) ($(a)+(150:\rs)$) };
		\draw plot [smooth cycle] coordinates {($(b)+(30:\rb)$) ($(b)+(90:\rs)$) ($(b)+(150:\rb)$) ($(b)+(210:\rs)$) ($(b)+(270:\rb)$) ($(b)+(330:\rs)$) };
		\begin{scope}[xshift = 1.2cm]
		\draw (-1.6,0) node{$\cdots$};
		\end{scope};
		\end{scope}
		\begin{scope}[xshift=2.3cm]
		\coordinate (a) at (0.25,0);		
		\coordinate (b) at (1.05,0);	
		\draw plot [smooth cycle] coordinates {($(a)+(210:\rb)$) ($(a)+(270:\rs)$) ($(a)+(330:\rb)$) ($(a)+(30:\rs)$) ($(a)+(90:\rb)$) ($(a)+(150:\rs)$) };
		\draw plot [smooth cycle] coordinates {($(b)+(30:\rb)$) ($(b)+(90:\rs)$) ($(b)+(150:\rb)$) ($(b)+(210:\rs)$) ($(b)+(270:\rb)$) ($(b)+(330:\rs)$) };
		\begin{scope}[xshift = 1.1cm]
		\end{scope}
		\draw (1.7,0) node{$\cdots$};
		\end{scope}
		
		\begin{scope}[yshift=0.9cm]
		\draw [decorate,opacity=0.5,decoration={brace,amplitude=10pt}]
		(-0.5,0) -- (1.8,0) node [black,midway]{};
		\node[opacity=0.5] at (0.67,0.35) {$Y$};	
		\end{scope}
		\end{tikzpicture}
		\caption{Example of the array, drawn to illustrate the symmetry assumptions.} \label{fig:SSH} 
	\end{figure}

	We first describe the geometry of the structure, depicted in \Cref{fig:SSH}. We will consider a three-dimensional geometry which is periodic in one dimension.  Let $Y=[-L/2,L/2]\times \R^2$ be the unit cell, with half-cells $Y_1 = [-L/2,0]\times \R^2$ and $Y_2 = [0,L/2]\times \R^2$. For $j=1,2$, we assume that $Y_j$ contains a resonator $D_j$ such that $\p D_j$ is of Hölder class $C^{1,s}$ for some $0<s<1$. We denote a pair of resonators, a so-called dimer, by $D = D_1 \cup D_2$. We assume that the dimer is \emph{parity symmetric}, that is,
	\begin{equation} \label{eq:symmetry}
	\P D = D,
	\end{equation}
	where $\P$ is the parity operator $\P: \R^3 \rightarrow \R^3, \P(x) = -x$.
	
	The geometry under consideration is periodic in the direction specified by $\w := (1,0,0)$.	We define the translated resonators $D_i^m, i=1,2, m\in \Z,$ by $D_i^m = D_i + mL\w$ and the translated dimers by $D^m = D_1^m \cup D_2^m$. The total crystal $\C$ is given by
	$$\C = \bigcup_{m\in\Z} D^m.$$
	As found in \cite{ammari2020exceptional}, any imaginary part of the densities inside the resonators will not affect the resonant frequencies to leading order. Therefore, we assume that all the resonators have equal densities $\rho_b \in \R$, and are with different bulk modulus $\kappa_i^m \in \mathbb{C}$, where the imaginary part of $\kappa_i^m$ corresponds to the gain or loss inside the resonator. We define the parameters
	$$v_i^m = \sqrt{\frac{\kappa_i^m}{\rho_b}}, \quad v = \sqrt{\frac{\kappa}{\rho}}, \quad \delta = \frac{\rho_b}{\rho}, \quad k = \frac{\omega}{v}, \quad k_i^m = \frac{\omega}{v_i^m}.$$	
	We model acoustic wave propagation inside the structure by the Helmholtz problem
	\begin{equation} \label{eq:scattering}
	\left\{
	\begin{array} {ll}
	\ds \Delta {u}+ k^2 {u}  = 0 & \text{in } \R^3 \setminus \C, \\
	\nm
	\ds \Delta {u}+ (k_i^m)^2 {u}  = 0 & \text{in } D_i^m, \\
	\nm
	\ds  {u}|_{+} -{u}|_{-}  =0  & \text{on } \p \C, \\
	\nm
	\ds  \delta \frac{\partial {u}}{\partial \nu} \bigg|_{+} - \frac{\partial {u}}{\partial \nu} \bigg|_{-} =0 & \text{on } \p \C, \\
	\nm
	\ds u(x_1,x_2,x_3) & \text{satisfies the outgoing radiation condition as } \sqrt{x_2^2+x_3^2} \rightarrow \infty.
	\end{array}
	\right.
	\end{equation}
	In order to have resonant frequencies in the subwavelength regime, we will study the case of a high contrast in the density, corresponding to 
	$$\delta \ll 1.$$
	In the limit  $\delta \rightarrow 0$, we say a frequency $\omega$ (or corresponding eigenmode) is \emph{subwavelength} if $\omega$ scales as $O(\delta^{1/2})$. ´
	
	The geometry described by $\C$ is periodic, but due to the different values of $\kappa_i^m$ the differential problem \eqref{eq:scattering} is in general not periodic. In the following, we will study different realisations of \eqref{eq:scattering}. In \Cref{sec:periodic}, we study the periodic case, \ie{} when $\kappa_i^m$ does not depend on $m$. In \Cref{sec:first}, we study a case when an ``edge'' is introduced, giving localized edge modes. For completeness, in \Cref{sec:geomd} we numerically demonstrate the edge modes in a system with a defect in the geometry.
	
	\subsection{Layer potential theory} \label{sec:layerpot}
	We denote the (outgoing) Helmholtz Green's functions by $G^k$, defined by
	$$
	G^k(x,y) := -\frac{e^{\iu k|x-y|}}{4\pi|x-y|}, \quad x,y \in \R^3, x\neq y, k\in \mathbb{C}.
	$$
	Let $D\in \R^3$ be a bounded, multiply connected domain with $N$ simply connected components $D_i$. Further, suppose that there exists some $0<s<1$ so that $\p D_i$ is of Hölder class $C^{1,s}$ for each $i=1,\ldots,N$. 
	
	We introduce the single layer potential $\S_{D}^k: L^2(\partial D) \rightarrow H_{\textrm{loc}}^1(\R^3)$, defined by
	\begin{equation*}
	\S_D^k[\phi](x) := \int_{\partial D} G^k(x,y)\phi(y) \dx \sigma(y), \quad x \in \R^3.
	\end{equation*}
	Here, the space $H_{\textrm{loc}}^1(\R^3)$ consists of functions that are square integrable and with a square integrable weak first derivative, on every compact subset of $\R^3$. Taking the trace on $\p D$, it is well-known that $\S_D^0: L^2(\p D) \rightarrow H^1(\p D)$ is invertible.
	
	We also define the Neumann-Poincar\'e operator $\K_D^{k,*}: L^2(\partial D) \rightarrow L^2(\partial D)$ by
	\begin{equation*}
	\K_D^{k,*}[\phi](x) := \int_{\partial D} \frac{\partial }{\partial \nu_x}G^k(x,y) \phi(y) \dx \sigma(y), \quad x \in \partial D,
	\end{equation*}
	where $\partial/\partial \nu_x$ denotes the outward normal derivative at $x\in\p D$.
	
	The following relations, often known as \emph{jump relations}, describe the behaviour of $\S_D^k$ on the boundary $\partial D$ (see, for example, \cite{MaCMiPaP}):
	\begin{equation}
	\S_D^k[\phi]\big|_+ = \S_D^k[\phi]\big|_-,
	\end{equation}
	and
	\begin{equation}
	\frac{\partial }{\partial \nu}\S_D^k[\phi]\Big|_{\pm}  =  \left(\pm\frac{1}{2} I + \K_D^{k,*}\right) [\phi],
	\end{equation}
	where  $I$ is the identity operator and $|_\pm$ denote the limits from outside and inside $D$.

	% 	 $\S_D^k: L^2(\p D) \rightarrow H^1(\p D)$ is invertible when $k^2$ is not a Dirichlet eigenvalue of $D$. 
	%	Moreover, the following lemma characterizes the kernel of the operator $-\frac{1}{2} + \K_D^0$ \cite{arma2, MaCMiPaP}.
	
	% 	\begin{lem}\label{lem:kernel}
	% 		Let $N$ be the number of connected components of $D$. The functions $\psi_i, i=1,...,N$, defined as 
	% 		$$\psi_i := (\S_D^0)^{-1}[\chi_{\p D_i}],$$
	% 		where $\chi_{\p D_i}$ is the indicator function of $\p D_i$, is a basis for the kernel 
	% 		$$\ker\left(-\frac{1}{2}I + \K_D^0\right).$$
	% 		In particular, the kernel has dimension $N$.
	% 	\end{lem}

	\subsection{Floquet-Bloch theory and quasiperiodic layer potentials}\label{sec:floquet}
	A function $f(x)\in L^2(\R)$ is said to be $\alpha$-quasiperiodic if $e^{-\iu \alpha x}f(x)$ is a periodic function of $x$. If the periodicity is $L>0$, the quasiperiodicity $\alpha$ is defined modulo $\tfrac{2\pi}{L}$. Therefore, we define the \textit{first Brillouin zone} $Y^*$ as the torus $Y^*:= \R / \tfrac{2\pi}{L} \Z \simeq (-\pi/L, \pi/L]$. Given a function $f\in L^2(\R)$, the Floquet transform is defined as
	\begin{equation}\label{eq:floquet}
	\F[f](x,\alpha) := \sum_{m\in \Z} f(x-mL) e^{\iu \alpha mL}.
	\end{equation}
	$\F[f]$ is always $\alpha$-quasiperiodic in $x$ and periodic in $\alpha$. Let $Y_0 = [-L/2,L/2)$ be the one-dimensional unit cell. The Floquet transform is an invertible map $\F:L^2(\R) \rightarrow L^2(Y_0\times Y^*)$. The inverse is given by (see, for instance, \cite{MaCMiPaP,kuchment})
	\begin{equation*}
	\F^{-1}[g](x) = \frac{L}{2\pi}\int_{Y^*} g(x,\alpha) \dx \alpha, \quad x\in \R.
	\end{equation*}
	%, with inverse (see, for instance, \cite{MaCMiPaP,kuchment})$$ \F^{-1}[g](x) = \frac{1}{2\pi}\int_{Y^*} g(x,\alpha) \dx \alpha, \quad x\in \R,$$where $g(x,\alpha)$ is the quasiperiodic extension of $g$ for $x$ outside of the unit cell $Y_0$.
	
	We define the quasiperiodic Green's function $G^{\alpha,k}(x,y)$ as the Floquet transform of $G^k(x,y)$ along the direction specified by $\w$, \ie{},
	$$G^{\alpha,k}(x,y) := -\sum_{m \in \Z} \frac{e^{\iu k|x-y-mL\w|}}{4\pi|x-y-mL\w|}e^{\iu \alpha mL}.$$	
	Analogously to \Cref{sec:layerpot}, we define the quasiperiodic single layer potential $\mathcal{S}_D^{\alpha,k}$ by
	$$\mathcal{S}_D^{\alpha,k}[\phi](x) := \int_{\partial D} G^{\alpha,k} (x,y) \phi(y) \dx\sigma(y),\quad x\in \mathbb{R}^3.$$
	It is known that $\mathcal{S}_D^{\alpha,0} : L^2(\p D) \rightarrow H^1(\p D)$ is invertible if $\alpha \neq  0$ \cite{MaCMiPaP}, and for low frequencies we have
	\begin{equation}\label{eq:Sexp}
	\mathcal{S}_D^{\alpha,k} = \mathcal{S}_D^{\alpha,0} + O(k^2).
	\end{equation}
	Moreover, on the boundary $\p D$, $\S_D^{\alpha,k}$ satisfies the jump relations
	\begin{equation} \label{eq:jump1}
	\S_D^{\alpha,k}[\phi]\big|_+ = \S_D^{\alpha,k}[\phi]\big|_-,
	\end{equation}
	and
	\begin{equation} \label{eq:jump2}
	\frac{\p}{\p\nu} \S_D^{\alpha,k}[\phi] \Big|_{\pm} = \left( \pm \frac{1}{2} I +( \mathcal{K}_D^{-\alpha,k} )^*\right)[\phi],
	\end{equation}
	where $(\mathcal{K}_D^{-\alpha,k})^*$ is the quasiperiodic Neumann-Poincaré operator, given by
	$$ (\mathcal{K}_D^{-\alpha, k} )^*[\phi](x):= \int_{\p D} \frac{\p}{\p\nu_x} G^{\alpha,k}(x,y) \phi(y) \dx\sigma(y).$$
	
	% 	The quasiperiodic single layer potential $\mathcal{S}_D^{\alpha,0} : L^2(\p D) \rightarrow H^1(\p D)$ is invertible if $\alpha \neq  0$ \cite{MaCMiPaP}.
	
	\begin{rmk} \label{rmk:dim}
		To simplify the presentation, we have only defined the three-dimensional layer potentials and will perform the analysis in three spatial dimensions. However, we can analogously define the two-dimensional layer potentials \cite{MaCMiPaP}. Doing so, the analysis of \Cref{sec:periodic,sec:first} directly extend to the two-dimensional case, yielding the same conclusions.
	\end{rmk}
	
	\section{Periodic problem} \label{sec:periodic}
	In this section, we study the periodic problem, \ie{} when 
	\begin{equation*}
	\kappa_1^m = \kappa_1, \qquad \kappa_2^m = \kappa_2,
	\end{equation*}
	for all $m \in \Z$, for some bulk moduli $\kappa_i \in \mathbb{C}, \ i=1,2$. We also assume that $\Re(\kappa_1) = \Re(\kappa_2)$.
	Taking the Floquet transform of \eqref{eq:scattering}, we have
	\begin{equation} \label{eq:scattering_quasip}
	\left\{
	\begin{array} {ll}
	\ds \Delta {u^\alpha}+ k^2 {u^\alpha}  = 0 & \text{ in } Y \setminus D, \\
	\nm
	\ds \Delta {u^\alpha}+ k_i^2u^\alpha  = 0 & \text{ in } D_i, \\
	\nm
	\ds  {u^\alpha}|_{+} -{u^\alpha}|_{-}  =0  & \text{ on } \p D,\\
	\nm
	\ds  \delta \frac{\partial {u^\alpha}}{\partial \nu} \bigg|_{+} - \frac{\partial {u^\alpha}}{\partial \nu} \bigg|_{-} =0 & \text{ on } \p D,   \\
	\nm
	\ds u^\alpha(x+mL\w)= e^{\iu \alpha m} u^\alpha(x) & \text{ for all } m \in \Z, \\
	\nm
	\ds u^\alpha(x_1,x_2,x_3)& \text{ satisfies the $\alpha$-quasiperiodic outgoing radiation condition} \\ &\hspace{0.5cm} \text{as } \sqrt{x_2^2+x_3^2} \rightarrow \infty,
	\end{array}
	\right.
	\end{equation}
	where $u^\alpha(x) = \F[u](x,\alpha)$. The frequencies $\omega$ in the spectrum of \eqref{eq:scattering_quasip} are called \emph{quasiperiodic resonant frequencies}, and we say that the corresponding solution $u^\alpha$ is a (right) \emph{Bloch eigenmode}. Moreover, $\overline{\omega}$ will be in the spectrum of the system corresponding to
	\begin{equation*}
	\kappa_1^m = \overline{\kappa_1}, \qquad \kappa_2^m = \overline{\kappa_2},
	\end{equation*}
	and we say that the corresponding solution $v^\alpha$ is a \emph{left Bloch eigenmode}.
	
	We will refer to the case $\kappa_1, \kappa_2 \in \R$ as the \emph{Hermitian} case, and otherwise as the \emph{non-Hermitian}  case. We emphasise, however, that \eqref{eq:scattering_quasip} can be viewed as the spectral problem for an operator which, even in the case $\kappa_1, \kappa_2 \in \R$, is not self-adjoint (due to the radiation condition). The motivation for this terminology is that we will be able, using the \emph{capacitance matrix} formulation, to approximate the continuous spectral problem with a discrete eigenvalue problem which is Hermitian precisely in the case $\kappa_1, \kappa_2 \in \R$.
	
	\subsection{Complex band structure}
	In the periodic case, the spectrum $\sigma$ of \eqref{eq:scattering} can be decomposed into band functions $\omega_n^\alpha$, which are functions of  $\alpha \in Y^*, n=1,2,...$, by taking the Floquet transform:
	$$\sigma = \bigcup_{n=1}^\infty \bigcup_{\alpha\in Y^*} \omega_n^\alpha.$$	
	Since the material parameters are complex, the band functions $\omega_n^\alpha$ will in general be complex. Nevertheless, we define band gaps and degeneracies analogously to the case of real band function. Following \cite{shen2018topological}, we say that a band $\omega_n^\alpha$ is \emph{separable} if $\omega_n^\alpha \neq \omega_m^\alpha$ for all $m\neq n$ and all $\alpha \in Y^*$, and otherwise \emph{degenerate}. In this setting, a \emph{band gap} is a connected component of $\mathbb{C} \setminus \sigma$.
	
	Similarly to the previous works \cite{ammari2020exceptional,ammari2018honeycomb,ammari2019topologically}, the band structure and the eigenmodes can be approximated using a capacitance matrix formulation. Therefore, we let $V_j^\alpha$ be the solution to
	\begin{equation} \label{eq:V_quasi}
	\begin{cases}
	\ds \Delta V_j^\alpha =0 \quad &\mbox{in } \quad Y\setminus D,\\
	\ds V_j^\alpha = \delta_{ij} \quad &\mbox{on } \quad \partial D_i,\\
	\ds V_j^\alpha(x+mL\w)= e^{\iu \alpha m} V_j^\alpha(x) & \text{for all } m \in \Z, \\
	\ds V_j^\alpha(x_1,x_2,x_3) = O\left(\tfrac{1}{\sqrt{x_2^2+x_3^2}}\right) \quad &\text{as } \sqrt{x_2^2+x_2^2}\to\infty, \text{ uniformly in } x_1,
	\end{cases}
	\end{equation}
	where $\delta_{ij}$ is the Kronecker delta.
	We define the quasiperiodic capacitance coefficients $C_{ij}^\alpha$, for $i,j=1,2$, by
	\begin{equation} \label{eq:qp_capacitance} 
	C_{ij}^\alpha := \int_{Y\setminus D}\overline{\nabla V_i^\alpha}\cdot\nabla V_j^\alpha  \dx x,\quad i,j=1, 2,
	\end{equation}
	and then introduce the weighted quasiperiodic capacitance matrix $C^{v,\alpha}$ as
	\begin{equation}
	C^{v,\alpha}=\frac{1}{\rho}\begin{pmatrix}
	\kappa_1 C_{11}^\alpha & \kappa_1 C_{12}^\alpha \\[0.3em]
	\kappa_2 C_{21}^\alpha & \kappa_2 C_{22}^\alpha
	\end{pmatrix}.
	\end{equation}
	Observe that $\kappa_i/\rho = O(\delta)$, so the eigenvalues of $C^{v,\alpha}$ scale as $O(\delta)$. The following theorem was proved in \cite{ammari2020exceptional}.
	\begin{thm} \label{thm:res_quasi}
		As $\delta \rightarrow 0$, the quasiperiodic resonant frequencies satisfy the asymptotic formula
		$$\omega_i^\alpha = \sqrt{\frac{\lambda_i^\alpha}{|D_1|}} + O(\delta), \quad i = 1,2,$$
		where $|D_1|$ is the volume of a single resonator. Here, $\lambda_i^\alpha$ are the eigenvalues of the weighted quasiperiodic capacitance matrix $C^{v,\alpha}$.
	\end{thm}
	Observe that $|D_1|=|D_2|$ due to the $\P$-symmetry of the dimer $D$. The eigenvalues $\lambda_i^\alpha$ of $C^{v,\alpha}$ are given by
	$$\lambda_{j}^\alpha = \frac{1}{\rho}\left(C_{11}^\alpha\frac{\kappa_1+\kappa_2}{2} + (-1)^j \sqrt{\left(\frac{\kappa_1-\kappa_2}{2}\right)^2(C_{11}^\alpha)^2 + \kappa_1\kappa_2|C_{12}^\alpha|^2} \right).$$
	For a degeneracy to occur for small $\delta$, we need $\lambda_1^\alpha = \lambda_2^\alpha$ at some $\alpha \in Y^*$. It is straightforward to verify that this occurs precisely when $\kappa_1 = \overline{\kappa_2} := \kappa$ and $|\kappa| \geq \kappa_0$ for some $\kappa_0$ that depends on the geometry. These parameter values, \ie{} balanced and large enough gain/loss, were found in \cite{ammari2020exceptional} to correspond to an exceptional point. When the gain and loss are not balanced, \ie{} $\kappa_1 \neq \overline{\kappa_2}$, this degeneracy is lifted and the two bands are separable.
	
	In the case when the first and second bands are separable, we define the \emph{vorticity}, $\nu$, as 
	$$\nu = \frac{1}{2\pi}\int_{Y^*} \frac{\p}{\p \alpha} \arg\left(\omega_2^\alpha - \omega_1^\alpha\right) \dx \alpha.$$
	The vorticity is given by the winding number of $\omega_2^\alpha - \omega_1^\alpha$ around the origin, as $\alpha$ varies across the Brillouin zone $Y^*$.

	\begin{prop} \label{prop:vorticity}
		The vorticity $\nu$ vanishes for all $\delta$ small enough.
	\end{prop}
	\begin{proof}
	We will begin by proving that $\omega_j^\alpha = \omega_j^{-\alpha}$ for $j = 1,2$. We assume that we have a nonzero solution $u^\alpha$ to \eqref{eq:scattering_quasip} corresponding to $\omega_j^\alpha = \omega$. We set  $L^2(\p D) = L^2(\p D_1) \times L^2(\p D_2)$ and define $\hat{\S}_D^\omega$ and $\hat{\K}_D^{\omega,*}$ as
	$$\hat{\S}_D^\omega := \begin{pmatrix} \S_{D_1}^{k_1} & 0 \\ 0 & \S_{D_2}^{k_2} \end{pmatrix}, \qquad \hat{\K}_D^{\omega,*} := \begin{pmatrix} \K_{D_1}^{k_1,*} & 0 \\ 0 & \K_{D_2}^{k_2,*} \end{pmatrix}.$$ Using the single layer potentials, we can write $u^\alpha$ as 
	$$u^\alpha(x) = \begin{cases} \hat{\S}_D^\omega[\phi^\mathrm{in}](x), & x \in D, \\[0.3em] \S_D^{\alpha,k}[\phi^\mathrm{out}](x) & x\in Y \setminus \overline{D}, \end{cases}$$
	where $(\phi^\mathrm{in}, \phi^\mathrm{out}) \in L^2(\p D)$ satisfies the integral equation 
	$$\A^\alpha(\omega,\delta)\begin{pmatrix} \phi^\mathrm{in} \\ \phi^\mathrm{out}\end{pmatrix} = \begin{pmatrix} 0 \\ 0 \end{pmatrix}, \qquad \A^\alpha(\omega,\delta) := \begin{pmatrix} \hat{\S}_D^\omega & - \S_D^{\alpha,k} \\[0.3em] -\frac{1}{2}I + \hat{\K}_D^{\omega,*} & -\delta\left(\frac{1}{2}I + \K_D^{-\alpha,k,*}\right)\end{pmatrix}.$$

	Now, we define
	$$\psi^\mathrm{in} = \left(\hat{\S}_D^\omega\right) ^{-1}\P \hat{\S}_D^\omega[\phi^\mathrm{in}], \qquad \psi^\mathrm{out} = \phi^\mathrm{out},$$ 
	where, as before, $\P$ denotes the parity operator. Since $\S_{D_i}^0$ is invertible for $i=1,2$, and since $\omega$ scales as $O(\sqrt{\delta})$, $\psi^\mathrm{in}$ is well-defined for all $\delta$ small enough. Due to the $\P$-symmetry of $D$, we have 
	$$\S_D^{-\alpha,k} = \P\S_D^{\alpha,k},$$
	so it follows that 
	$$\A^{-\alpha}(\omega,\delta)\begin{pmatrix} \psi^\mathrm{in} \\ \psi^\mathrm{out}\end{pmatrix} = \begin{pmatrix} 0 \\ 0 \end{pmatrix}.$$
	Hence the function $u^{-\alpha}$, defined by 
	$$u^{-\alpha}(x) = \begin{cases} \hat{\S}_D^\omega[\psi^\mathrm{in}](x), & x \in D, \\[0.3em] \S_D^{-\alpha,k}[\psi^\mathrm{out}](x) & x\in Y \setminus \overline{D}, \end{cases}$$
	is a solution to \eqref{eq:scattering_quasip} at the quasiperiodicity $-\alpha$, corresponding to the frequency $\omega= \omega_j^{-\alpha}$. From this, we conclude that $\omega_j^{\alpha} = \omega_j^{-\alpha}$ for $j=1,2$, which implies that the winding number of $\omega_2^\alpha-\omega_1^\alpha$ vanishes with respect to any point in the complex plane.
	\end{proof}
	The \emph{skin effect} is a phenomenon where the system is highly sensitive to boundary conditions, and bulk modes can be localized. This has been linked to a non-zero vorticity, and
 	\Cref{prop:vorticity} suggests that the skin effect does not occur in the array of subwavelength resonators \cite{okuma2020topological}.
 	
	\Cref{fig:bandpt,fig:bandnpt} show the band structure in the cases of broken $\P\T$-symmetry and without $\P\T$-symmetry, respectively. In light of \Cref{rmk:dim}, we perform the computations in two spatial dimensions. Here, and throughout this work, the simulations were performed on circular resonators with unit radius and resonator separations $d = 0.5$ (within the unit cell) and $d'=6$ (between the unit cells). Moreover, the parameter values $\kappa =7000, \rho = 7000$ and $\rho_b=1$ are used throughout. The computations were performed using the multipole method as described in \cite{ammari2017subwavelength,ammari2018honeycomb}.
	
	As proven in \Cref{prop:vorticity}, when $\alpha$ varies from $-\pi/L$ to $\pi/L$, the frequencies $\omega^\alpha$ will initially (for $\alpha \in [-\pi/L,0]$) trace a curve in $\Cb$, and afterwards (for $\alpha \in [0,\pi/L]$) retrace the same curve with opposite orientation. Therefore, as illustrated in \Cref{fig:bandpt,fig:bandnpt}, the vorticity vanishes. %\todo{verify kappas}
	
	\begin{figure}[h]	
		\begin{center}
			\begin{subfigure}[b]{0.45\linewidth} 
				\centering
				\includegraphics[width=0.9\linewidth]{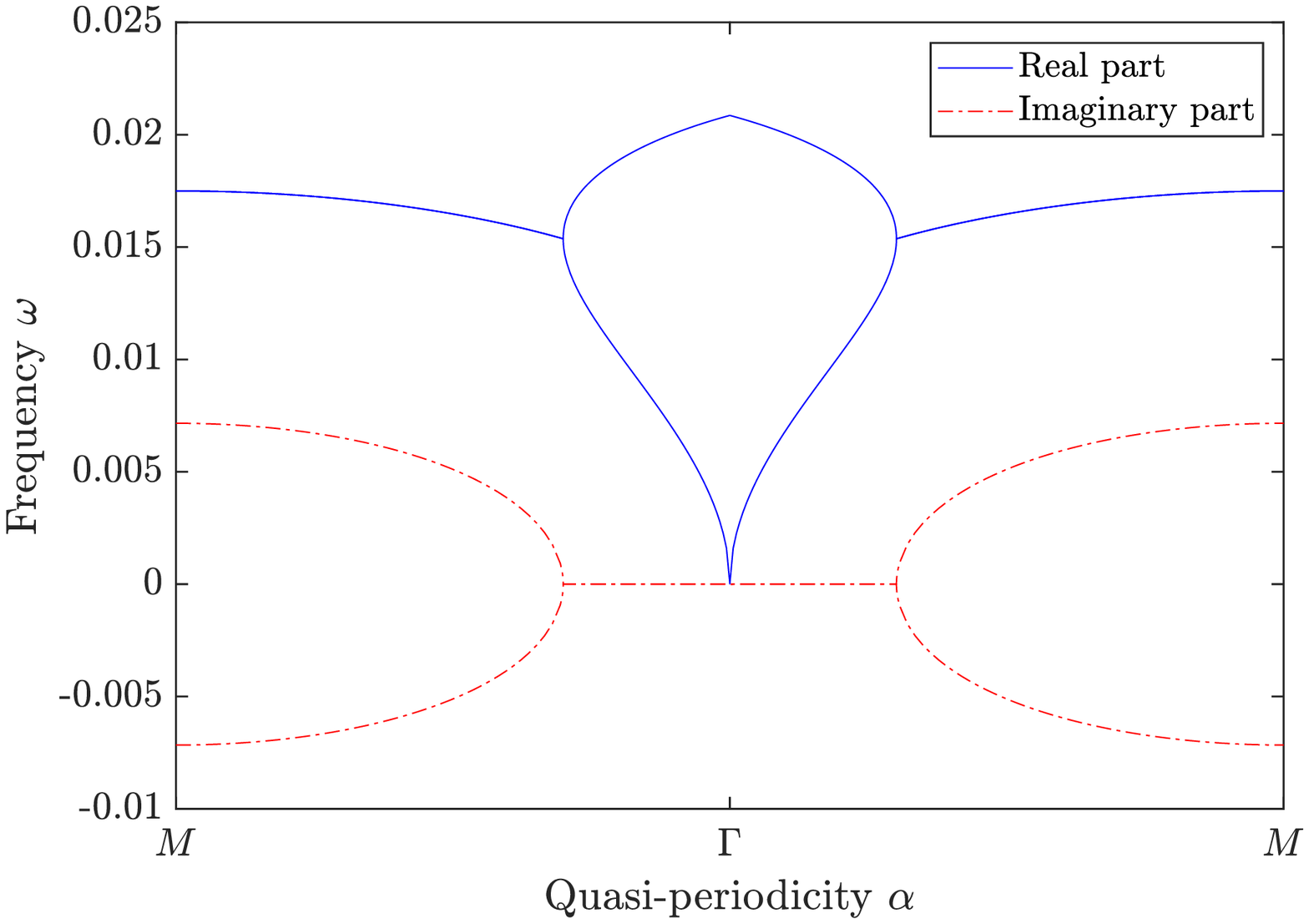}
				\caption{Band functions as function of the quasiperiodicity.}%\\ \hspace{1pt}}
			\end{subfigure}
			\hspace{0.6cm}
			\begin{subfigure}[b]{0.45\linewidth}
				\centering
				\includegraphics[width=0.9\linewidth]{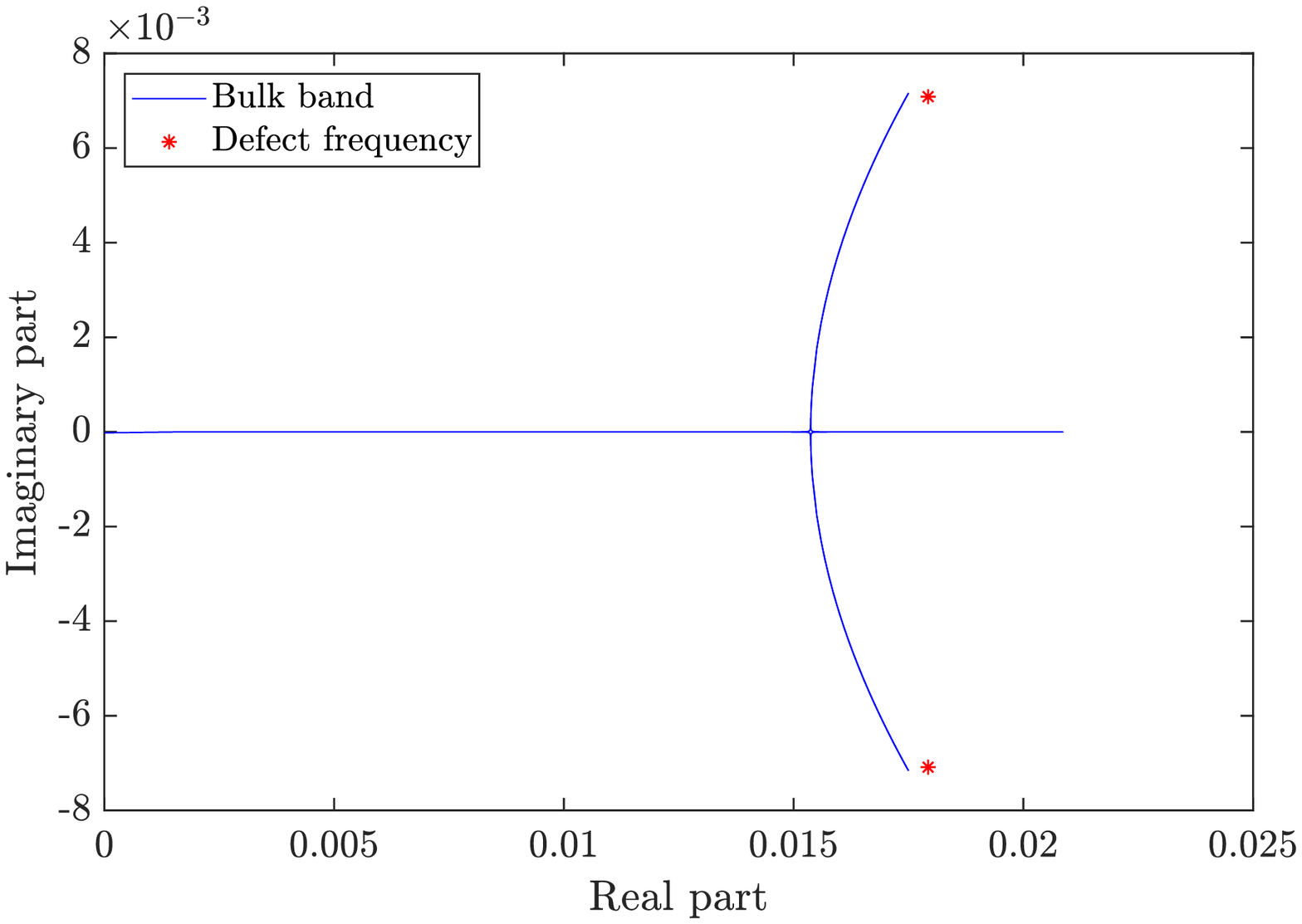}
				\caption{Trace of the band functions in the complex plane and defect frequency.}\label{fig:tracebpt}
			\end{subfigure}
		\end{center}
		\caption{Band structure in the case of a periodic micro-structure with broken $\P\T$-symmetry. Moreover, the frequencies of localized modes studied in \Cref{sec:first} are shown in \Cref{fig:tracebpt}. Here, we use the parameter values $\kappa_1 = 1 + 1.4\iu$ and $\kappa_2 = 1-1.4\iu$.} \label{fig:bandpt}
	\end{figure}

	\begin{figure}[h]	
	\begin{center}
		\begin{subfigure}[b]{0.45\linewidth} 
			\centering
			\includegraphics[width=0.9\linewidth]{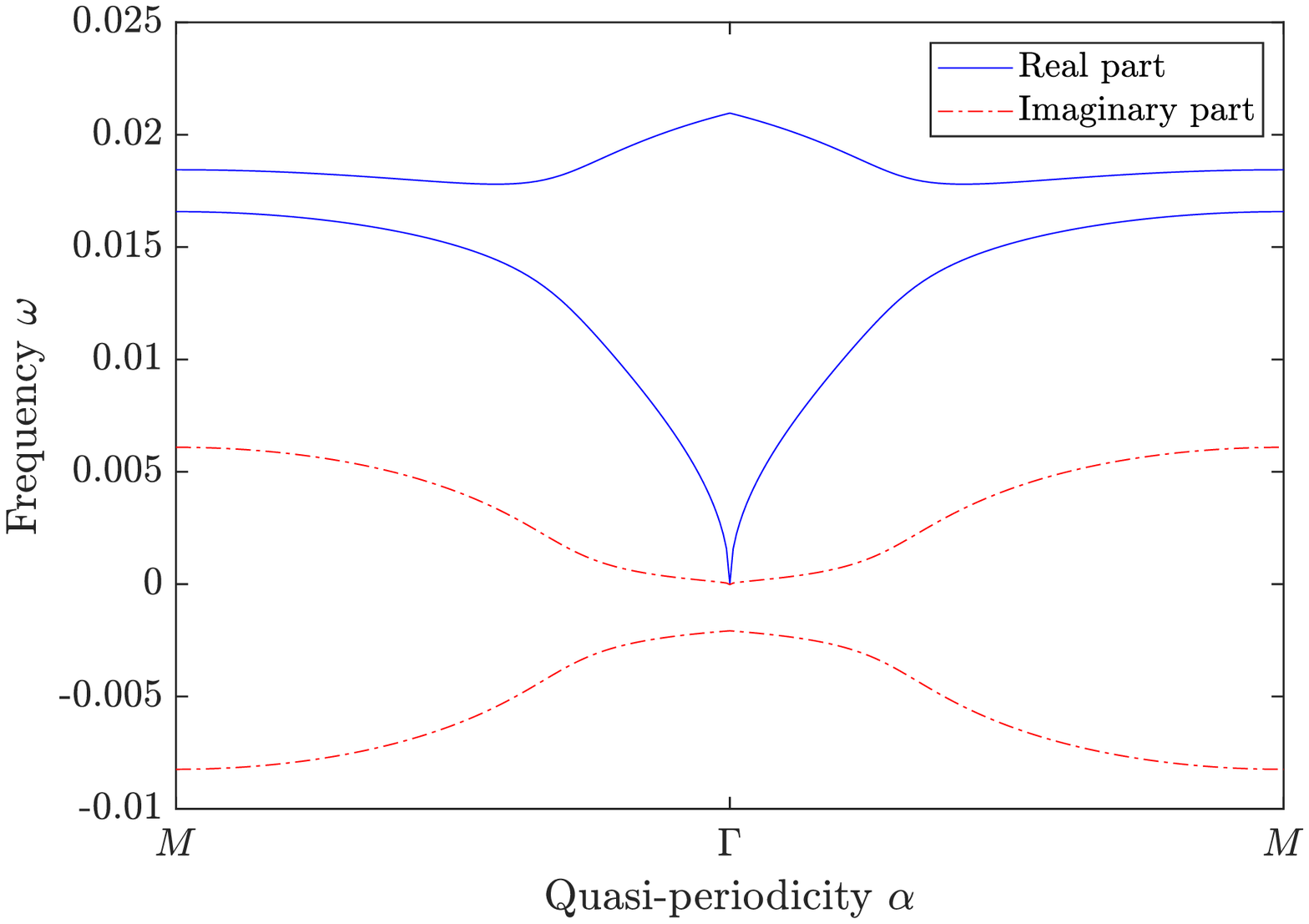}
			\caption{Band functions as function of the quasiperiodicity.}%\\ \hspace{1pt}}
		\end{subfigure}
		\hspace{0.6cm}
		\begin{subfigure}[b]{0.45\linewidth}
			\centering
			\includegraphics[width=0.9\linewidth]{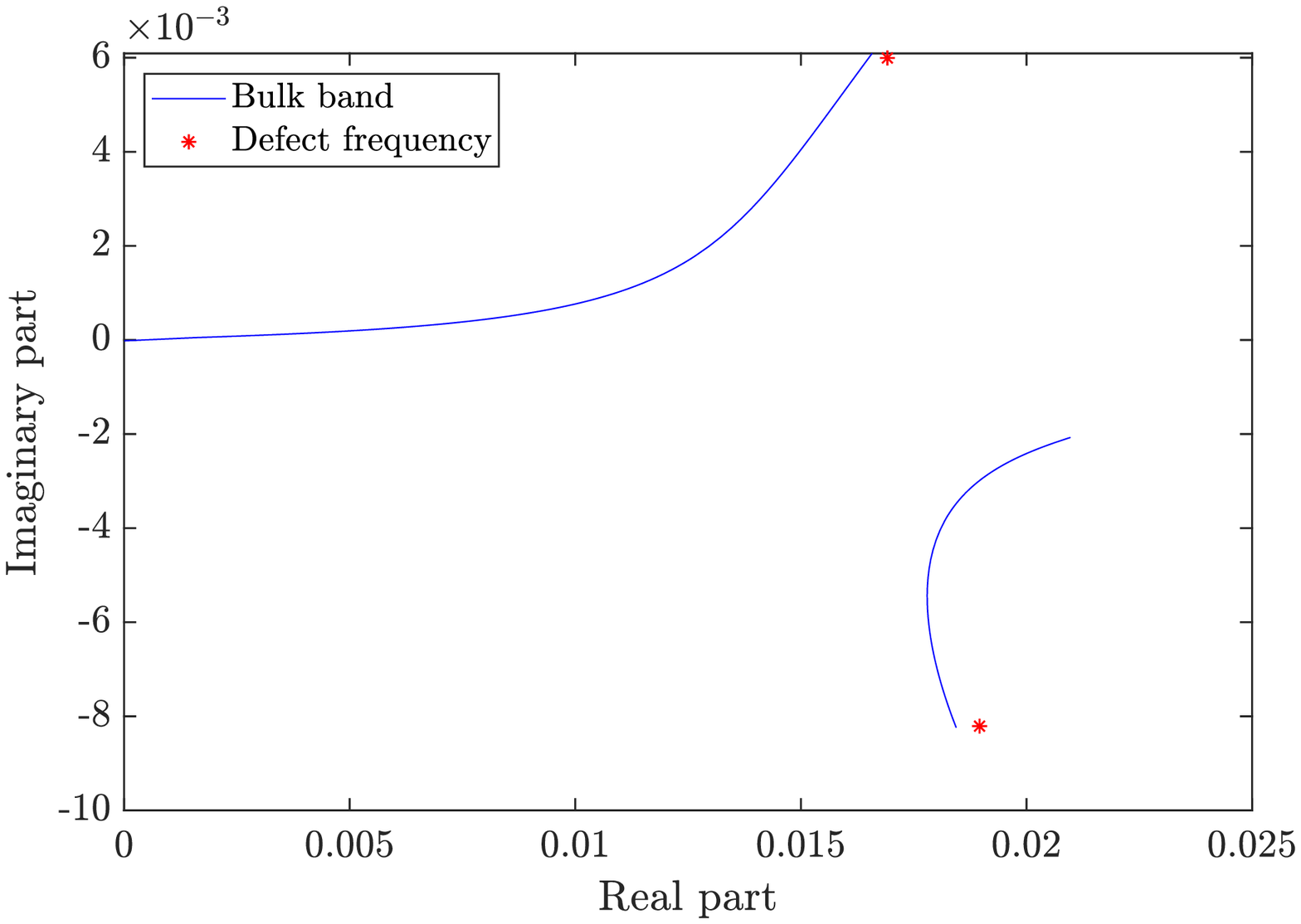}
			\caption{Trace of the band functions in the complex plane and defect frequency.} \label{fig:tracenpt}
		\end{subfigure}
	\end{center}
	\caption{Band structure in the case of a periodic micro-structure without $\P\T$-symmetry.  Moreover, the frequencies of localized modes studied in \Cref{sec:first} are shown in \Cref{fig:tracenpt}.  Here, we use $\kappa_1 = 1+1.2\iu, \kappa_2 = 1-1.6\iu$} \label{fig:bandnpt}
	\end{figure}
	
	\subsection{Non-Hermitian band inversion}
	In the case of real material parameters, it is well-known that \emph{band inversion} can occur, \ie{} that the monopole/dipole nature of the eigenmodes are swapped as $\alpha$ varies across the Brillouin zone (this has been demonstrated in the setting of subwavelength resonators in \cite{ammari2019topologically}). The band inversion is characterised by the so-called \emph{Zak phase}. In this section, we study a non-Hermitian generalization of the Zak phase, and demonstrate how band inversion can occur in non-Hermitian systems.
	
	Since $C^{v,\alpha}$ is non-Hermitian, the left and right eigenvectors do not coincide. For $j = 1,2,$ we let $\u_j = \left( \begin{smallmatrix} \u_{j,1} \\ \u_{j,2} \end{smallmatrix} \right)$ and $\v_j = \left( \begin{smallmatrix} \v_{j,1} \\ \v_{j,2} \end{smallmatrix} \right)  \ $  denote the eigenvectors of $C^{\alpha,v}$ and $\left(C^{\alpha,v}\right)^*$, respectively.
	
	\begin{comment}
	\begin{lem}
	An orthonormal system of eigenvectors $\u_j, \v_j, \ j=1,2$, \ie{} a system satisfying $\langle \v_i, \u_j \rangle = \delta_{ij}$, is given by
	\begin{align*}
	\u_j &= \frac{1}{i\sqrt{2\kappa_1\kappa_2}|C_{12}^{\alpha}|r}\begin{pmatrix} \ds\frac{\kappa_1-\kappa_2}{2}\kappa_1 C_{11}^\alpha C_{12}^\alpha - \kappa_1\kappa_2|C_{12}^\alpha|^2 \pm \kappa_1 C_{12}^\alpha r \\[0.3em] \ds \frac{\kappa_1-\kappa_2}{2}\kappa_2 C_{11}^\alpha \overline{C_{12}^\alpha} + \kappa_1\kappa_2|C_{12}|^2 \mp \kappa_2 \overline{C_{12}^\alpha}r\end{pmatrix}, \\
	\overline{\v_j} &= \frac{1}{i\sqrt{2\kappa_1\kappa_2}|C_{12}^{\alpha}|r}\begin{pmatrix} \ds\frac{{\kappa_2}-{\kappa_1}}{2}{\kappa_2} C_{11}^\alpha \overline{C_{12}^\alpha} + {\kappa_1\kappa_2}|C_{12}^\alpha|^2 \mp {\kappa_2} \overline{C_{12}^\alpha} {r} \\[0.3em] \ds \frac{{\kappa_2}-{\kappa_1}}{2}{\kappa_1} C_{11}^\alpha {C_{12}^\alpha} - {\kappa_1\kappa_2}|C_{12}|^2 \pm {\kappa_1} {C_{12}^\alpha}{r}\end{pmatrix}.
	\end{align*}
	Here, $r = \sqrt{\left(\frac{\kappa_1-\kappa_2}{2}\right)^2(C_{11}^\alpha)^2 + \kappa_1\kappa_2|C_{12}^\alpha|^2}$.
	\end{lem}
	\end{comment}
	
	\begin{lem}  \label{lem:modes}
		A bi-orthogonal system of eigenvectors $\u_j, \v_j,$ for  $j=1,2$, \ie{}, a system satisfying $\langle \v_i, \u_j \rangle = \delta_{ij}$, is given by
		$$\u_j = \frac{1}{\sqrt{2}}\begin{pmatrix} \ds e^{-\iu \phi_j} \\[0.3em] 1\end{pmatrix}, \qquad 
		\overline{\v_j} = \frac{1}{\sqrt{2}}\begin{pmatrix} \ds e^{\iu \theta_j^{(1)}} \\ e^{\iu \theta_j^{(2)}}\end{pmatrix},$$
		where the complex phases $\phi_j, \theta_j^{(1)},$ and $\theta_j^{(2)}$ are defined by
		\begin{align*}
		&e^{\iu (\theta_j^{(1)}-\theta_j^{(2)} + \phi_j)} = \frac{\kappa_2\overline{C_{12}^\alpha}}{\kappa_1C_{12}^\alpha},
		\qquad \qquad
		e^{\iu (\theta_j^{(1)}-\phi_j)} + e^{\iu \theta_j^{(2)}} = 2, \\
		&e^{-\iu \phi_j} = \frac{C_{11}^\alpha (\kappa_1-\kappa_2) + (-1)^j \sqrt{\left(\kappa_1-\kappa_2\right)^2(C_{11}^\alpha)^2 + 4\kappa_1\kappa_2|C_{12}^\alpha|^2}}{2\kappa_2\overline{C_{12}^\alpha}}.
		\end{align*}
		%\begin{align*}
		%\u_j &= \frac{1}{\sqrt{2}}\begin{pmatrix} \ds \frac{C_{11}^\alpha (\kappa_1-\kappa_2) \pm \sqrt{\left(\kappa_1-\kappa_2\right)^2(C_{11}^\alpha)^2 + 4\kappa_1\kappa_2|C_{12}^\alpha|^2}}{2\kappa_2\overline{C_{12}^\alpha}} \\[0.3em] 1\end{pmatrix}, \\
		%\overline{\v_j} &= \frac{1}{\sqrt{2}}\begin{pmatrix} \ds \frac{C_{11}^\alpha (\kappa_1-\kappa_2) \pm \sqrt{\left(\kappa_1-\kappa_2\right)^2(C_{11}^\alpha)^2 + 4\kappa_1\kappa_2|C_{12}^\alpha|^2}}{2\kappa_1C_{12}^\alpha} \\[0.3em] 1\end{pmatrix}.
		%\end{align*}
	\end{lem}
	We define the functions $S_j^\alpha$ by
	$$S_j^\alpha(x) := \begin{cases} \frac{1}{\sqrt{|D_1|}}\delta_{ij} \quad &x \in D_i, \ i=1,2, \\ \frac{1}{\sqrt{|D_1|}}V_j^\alpha(x) \quad &x \in Y\setminus D. \end{cases}$$
	These functions are the normalized extensions of $V_j^\alpha$ (defined in \eqref{eq:V_quasi}), in the sense that $\langle S_i, S_j \rangle = \delta_{ij}$. Here, and in the remainder of this work, $\langle \cdot, \cdot\rangle$ denotes the inner product in $L^2(D)$.
	
	The following approximation	result is a straightforward generalization of results from \cite{ammari2018honeycomb,ammari2019topologically}.
	\begin{lem}
		As $\delta \rightarrow 0$, we have the following approximation of the right and left Bloch eigenmodes:
		\begin{align*}
		u_j^\alpha &= \u_{j,1} S_1^\alpha + \u_{j,2} S_2^\alpha + O(\delta^{1/2}), \\
		v_j^\alpha &= \v_{j,1} S_1^\alpha + \v_{j,2} S_2^\alpha + O(\delta^{1/2}).
		\end{align*}
	\end{lem}
	We then define the \textit{(non-Hermitian) Zak phase}, $\varphi_j^{\mathrm{zak}}$, by \cite{ghatak2019new} 
	\begin{align*}
	\varphi_j^{\mathrm{zak}} &:= \frac{\iu}{2} \int_{Y^*}\left( \Big\langle v_j^\alpha, \frac{\p u_j^\alpha  }{\p \alpha}\Big\rangle + \Big\langle u_j^\alpha, \frac{\p  v_j^\alpha}{\p \alpha} \Big\rangle \right) \dx \alpha.
	\end{align*}
	In the case $\kappa_1, \kappa_2\in \R$, this definition coincides with the definition used in \cite{ammari2019topologically}. In the sequel, we will occasionally write $\varphi_j^{\mathrm{zak}}(\kappa_1, \kappa_2)$ to denote the Zak phase corresponding to the bulk modulus $\kappa_1$ inside $D_1$ and $\kappa_2$ inside $D_2$. %\todo{define invariant as leading order zak phase}
	
	A $2\times 2$ matrix $A = A(\alpha)$ is said to be \emph{chirally symmetric} if $A$ can be written 
	$$A(\alpha) = f(\alpha) I + B(\alpha),$$
	for some real function $f$ and some off-diagonal matrix $B$. In the case of the weighted quasiperiodic capacitance matrix, we have $C_{11}^\alpha = C_{22}^\alpha \in \R$ \cite{ammari2019topologically}, so $C^{v,\alpha}$ is chirally symmetric precisely in the case $\kappa_1=\kappa_2\in \R$. Non-Hermitian systems with chiral symmetry are known to have quantized Zak phases, which is not the case without chiral symmetry \cite{jiang2018topological,yin2018geometrical,liang2013topological}.
	\begin{lem} \label{lem:phase}
		The Zak phase $\varphi_j^{\mathrm{zak}}, j = 1,2,$ can be written as
		$$\varphi_j^{\mathrm{zak}} = -\mathrm{Im}\left( \int_{Y^*}\Big\langle \v_j,\frac{\p \u_j}{\p \alpha} \Big\rangle \dx \alpha \right) + O(\delta).$$
	\end{lem}
	\begin{proof}
		Observe that
		$$  \langle S_1^\alpha,  S_1^\alpha \rangle = 1, \qquad  \langle S_2^\alpha,  S_2^\alpha \rangle = 1, \qquad \langle S_1^\alpha,  S_2^\alpha \rangle = 0,$$
		and in $D$ we have 
		$$ \frac{\p}{\p \alpha}  S_1^\alpha \equiv 0, \qquad \frac{\p}{\p \alpha}  S_2^\alpha \equiv 0,$$
		for all $\alpha \in Y^*$. We then have
		\begin{align*}
		\Big\langle v_j^\alpha, \frac{\p  u_j^\alpha}{\p \alpha} \Big\rangle  &= \overline{\v_{j,1}}\frac{\p \u_{j,1}}{\p \alpha} + \overline{\v_{j,2}}\frac{\p \u_{j,2}}{\p \alpha} + O(\delta) \\
		&=\Big\langle \v_j,\frac{\p \u_j}{\p \alpha} \Big\rangle  + O(\delta).
		\end{align*}
		Moreover, from the normalization $\langle v_j^\alpha,u_j^\alpha \rangle =1$ we have
		\begin{align*}
		\Big\langle u_j^\alpha, \frac{\p  v_j^\alpha}{\p \alpha} \Big\rangle = - \overline{\Big\langle v_j^\alpha, \frac{\p u_j^\alpha  }{\p \alpha}\Big\rangle }.
		\end{align*}
		Combining the above approximations, we have
		$$\Big\langle v_j^\alpha, \frac{\p u_j^\alpha  }{\p \alpha}\Big\rangle + \Big\langle u_j^\alpha, \frac{\p  v_j^\alpha}{\p \alpha} \Big\rangle = \iu \, \mathrm{Im}\left( \Big\langle \v_j,\frac{\p \u_j}{\p \alpha} \Big\rangle \right) + O(\delta).$$
		Substituting this approximation into the definition of the Zak phase, we obtain the sought expression for $\varphi_j^{\mathrm{zak}}$ as $\delta \rightarrow 0$.
	\end{proof}
	%In view of the asymptotic formula of $\varphi_j^\mathrm{zak}$ in \Cref{lem:phase}, we will restrict the attention to the leading order of the Zak phase:	$$\varphi_j^{\mathrm{zak},0} := -\mathrm{Im}\left( \int_{Y^*}\Big\langle \v_j,\frac{\p \u_j}{\p \alpha} \Big\rangle \dx \alpha \right).$$

	For the next result, we will assume that the Hermitian counterpart of the structure is \emph{topologically trivial}. In other words, we assume
	\begin{equation}\label{eq:trivial}
	\varphi_j^{\mathrm{zak}}(\Re(\kappa_1), \Re(\kappa_2)) = 0.
	\end{equation}
	As shown in \cite{ammari2019topologically}, this can for example be achieved by having a dilute dimerized array, where the resonator separation is smaller within the unit cell compared to between the cells.
	\begin{prop} \label{prop:zak}
		Assume that the structure satisfies \eqref{eq:trivial} and that $\kappa_1, \kappa_2$ are chosen such that the first two band functions are separable. Then we have
		$$\varphi_j^{\mathrm{zak}}(\kappa_1, \kappa_2) = - \varphi_j^{\mathrm{zak}}(\kappa_2, \kappa_1) + O(\delta) \quad \text{and} \quad \varphi_j^{\mathrm{zak}}(\overline{\kappa_1}, \overline{\kappa_2}) = \varphi_j^{\mathrm{zak}}(\kappa_1, \kappa_2) + O(\delta).$$
		In particular, if $\kappa_1 = \overline{\kappa_2} = \kappa$, we have $\varphi_j^{\mathrm{zak}}(\kappa, \overline{\kappa}) = O(\delta).$
	\end{prop}
	\begin{proof}
		From Lemmas \ref{lem:modes} and \ref{lem:phase}, it is straightforward to show that $\varphi_1^{\mathrm{zak}}(\kappa_1,\kappa_2) = \varphi_2^{\mathrm{zak}}(\kappa_2,\kappa_1) + O(\delta)$ and that $\varphi_1^{\mathrm{zak}}(\overline{\kappa_1}, \overline{\kappa_2}) = -\varphi_2^{\mathrm{zak}}(\kappa_1, {\kappa_2})+ O(\delta)$.
		It is well-known that the ``total Zak phase'' $\varphi_{\mathrm{tot}}^{\mathrm{zak}} := \varphi_1^{\mathrm{zak}} + \varphi_2^{\mathrm{zak}}$ can only attain discrete multiples of $\pi$ \cite{jiang2018topological}. Since $\varphi_{\mathrm{tot}}^{\mathrm{zak}} $ is continuous for $\kappa_1, \kappa_2 \in \mathbb{C}$, it follows from \eqref{eq:trivial} that $\varphi_{\mathrm{tot}}^{\mathrm{zak}} = 0$ for all $\kappa_1, \kappa_2 \in \mathbb{C}$. Consequently, $\varphi_1^{\mathrm{zak}} = -\varphi_2^{\mathrm{zak}}$, and it follows that $\varphi_j^{\mathrm{zak}}(\kappa_1, \kappa_2) = - \varphi_j^{\mathrm{zak}}(\kappa_2, \kappa_1) + O(\delta)$ and $\varphi_j^{\mathrm{zak}}(\overline{\kappa_1}, \overline{\kappa_2}) = \varphi_j^{\mathrm{zak}}(\kappa_1, \kappa_2) + O(\delta)$. 
		%\todo{can the second equality be made exact? We should have $\varphi_j^{\mathrm{zak}}(\kappa, \overline{\kappa}) = 0.$}
		\end{proof}
	\begin{rmk}
		\Cref{prop:zak} provides intuition on how to create structures supporting edge modes. In the Hermitian case, the bulk-boundary correspondence indicates that edge modes will exist when joining two materials with distinct Zak phases. Unlike the Hermitian case, the non-Hermitian Zak phase is not quantized. \Cref{prop:zak} shows that distinct Zak phases can, in general, also be achieved by swapping $\kappa_1$ and $\kappa_2$ while keeping $d$ fixed. Clearly, this is a purely non-Hermitian effect, which disappears in the Hermitian limit as $\mathrm{Im}(\kappa_1), \mathrm{Im}(\kappa_2) \rightarrow 0$. 
	\end{rmk}
\begin{rmk}
	The total Zak phase $\varphi_{\mathrm{tot}}^{\mathrm{zak}}$ is known to be quantized and, under assumption \eqref{eq:trivial}, vanishes even for complex $\kappa_1, \kappa_2$. As we shall see, however, this quantized number fails to predict the existence of edge modes. In the Hermitian case, a nonzero Zak phase is equivalent to an ``inverted'' band structure. The fact that the Zak phase is quantized originates from the fact that the eigenmodes are purely monopole and dipole modes at $\alpha = 0$ and $\alpha = \pi/L$. A  non-integer value of the Zak phase can be attributed to a ``partial'' band inversion, where the eigenmodes wind around some point in the complex plane, and (due to non-Hermiticity) are expressed as (complex) linear combinations of monopole and dipole modes. Swapping the values of $\kappa_1$ and $\kappa_2$ swaps the sign of $\varphi_j^\mathrm{zak}$, corresponding to a reverse winding.
\end{rmk}

	\Cref{fig:invsym,fig:inv} demonstrate the non-Hermitian band inversion in the case of low and high gain/loss, respectively. Here, the phase factor of the eigenmodes, $e^{-\iu\phi_j}$, is shown in the complex plane as $\alpha$ varies across the Brillouin zone $Y^*$. In the cases without $\P\T$-symmetry, \ie{}, $\kappa_1\neq \overline{\kappa_2}$, the phase factor has a nonzero winding around some point in the complex plane, measured by the Zak phase. Due to non-Hermiticity, the phase factor does not vary between $-1$ and $1$, resulting in a non-quantized Zak phase. 
	
	In the case of unbroken $\P\T$-symmetry, the phase factor is confined to the unit circle and has zero winding with respect to any point in $\mathbb{C}$ (\Cref{fig:inv0sym}). This corresponds to zero Zak phase. In the case of broken $\P\T$-symmetry, the two bands are degenerate and the Zak phase is undefined. Nevertheless, the phases are no longer confined to the unit circle, and show a similar behaviour to the cases without $\P\T$-symmetry (\Cref{fig:inv0}). %\todo{add numerical values for Zak phase} %\todo{decide between $\phi_j$ and $e^{-\iu \phi_j}$}

	\begin{figure}[h]	
	\begin{center}
		\begin{subfigure}[b]{0.3\linewidth} 
			\centering
			\includegraphics[width=0.9\linewidth]{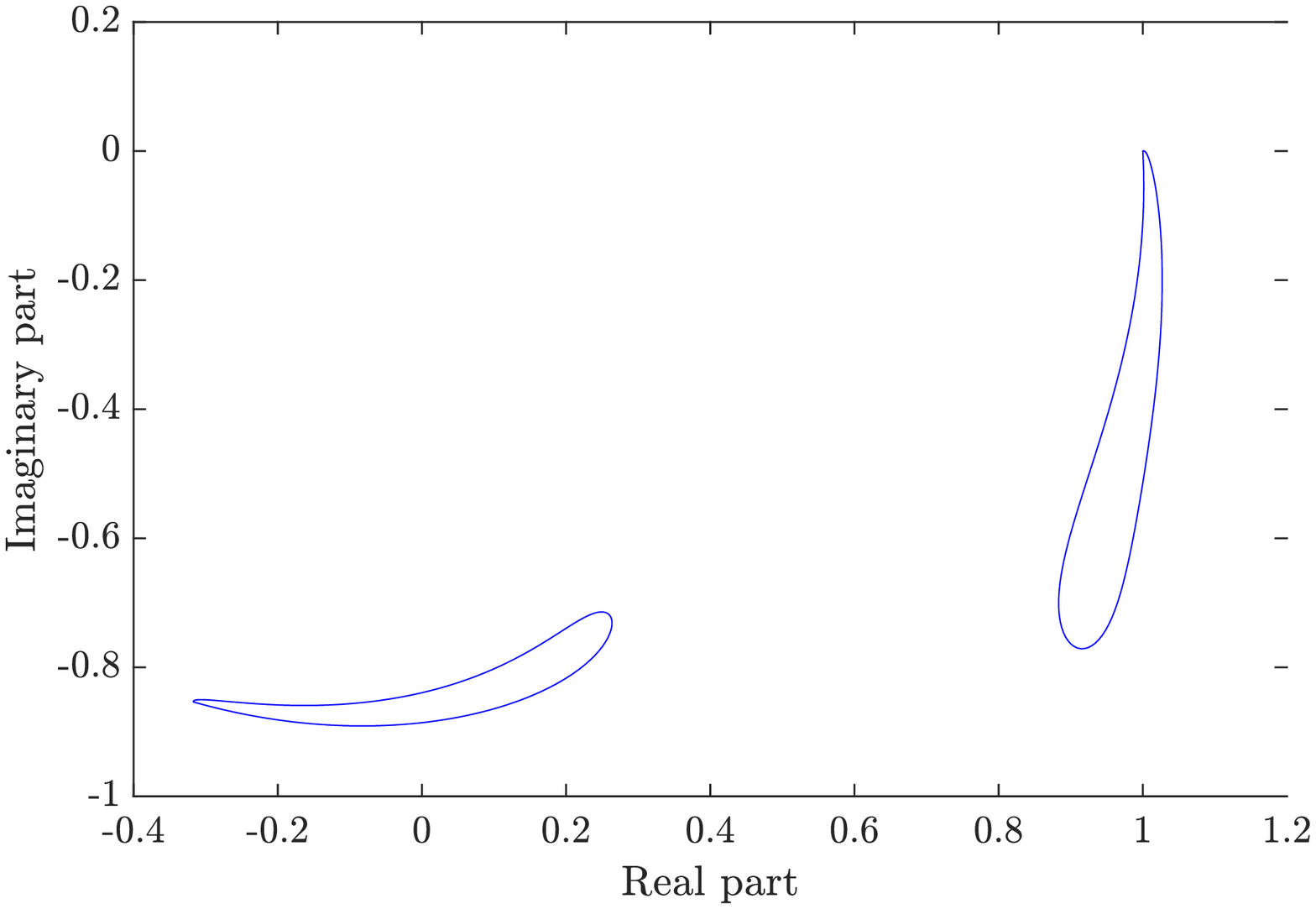}
			\caption{$\kappa_1 = 1+0.8\iu, \kappa_2 = 1-0.6\iu$}
		\end{subfigure}
		\hspace{0.1cm}
		\begin{subfigure}[b]{0.3\linewidth}
			\centering
			\includegraphics[width=0.9\linewidth]{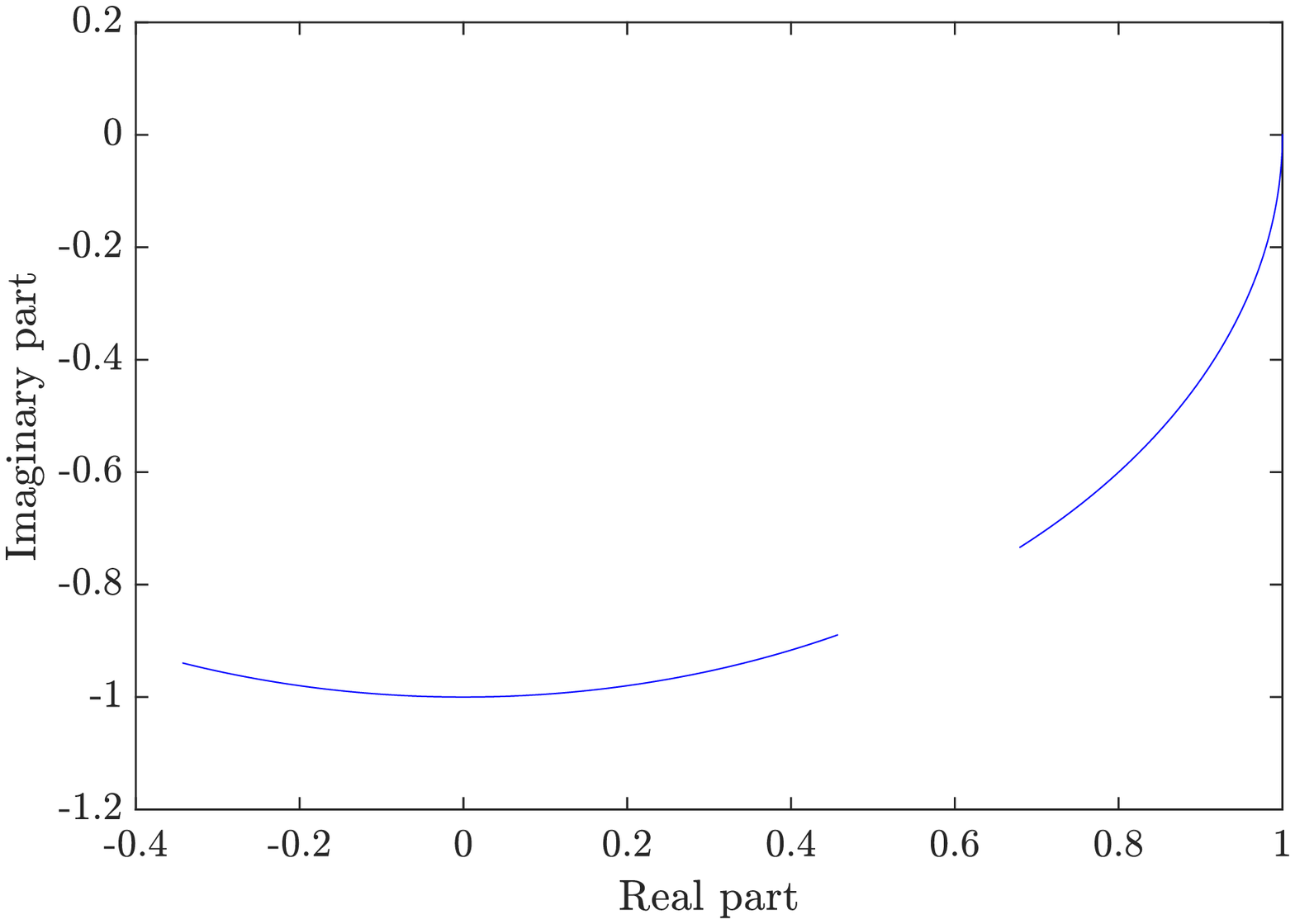}
			\caption{$\kappa_1 = 1+0.7\iu, \kappa_2 = 1-0.7\iu$} \label{fig:inv0sym}
		\end{subfigure}
		\hspace{0.1cm}
		\begin{subfigure}[b]{0.3\linewidth}
			\centering
			\includegraphics[width=0.9\linewidth]{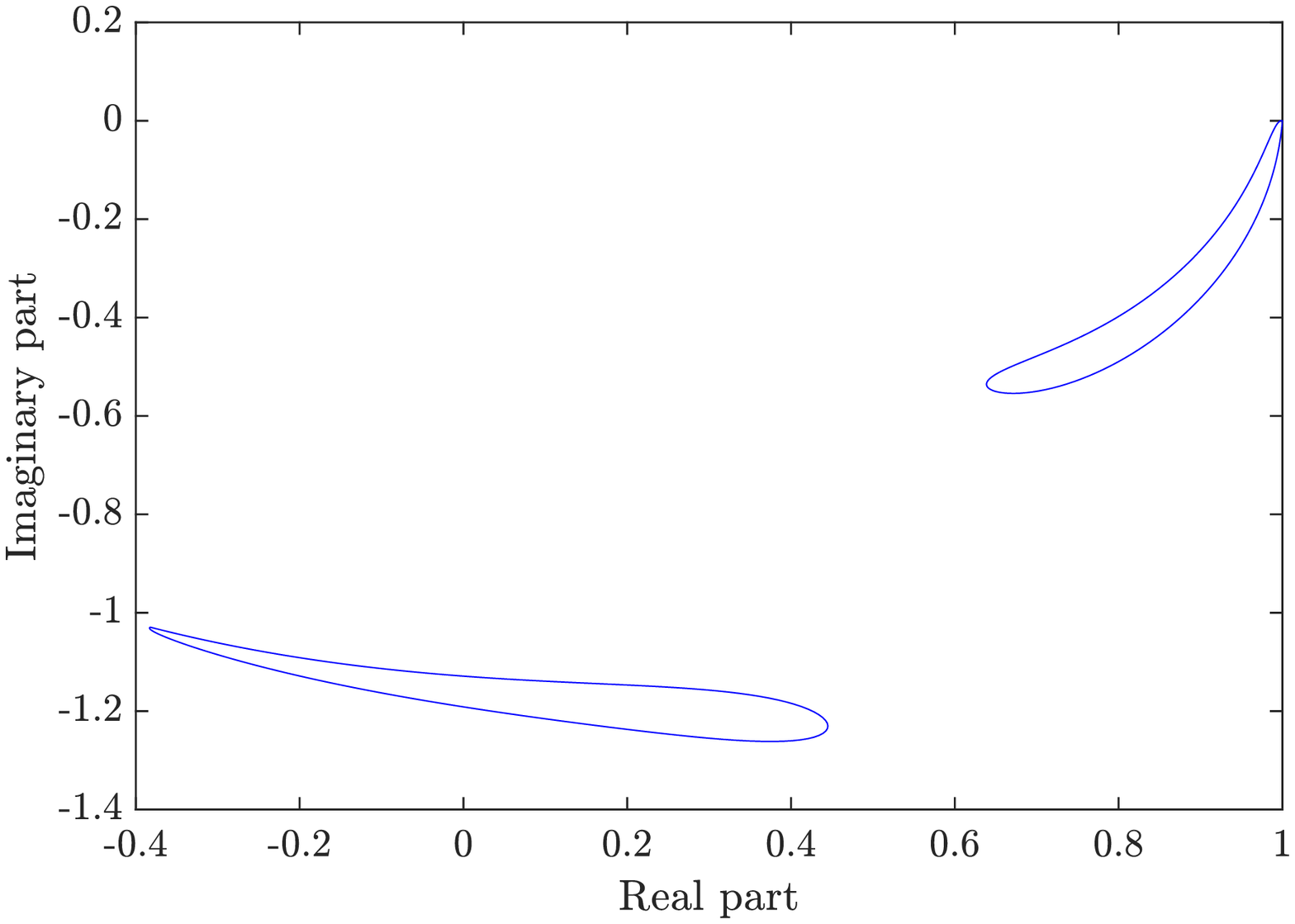}
			\caption{$\kappa_1 = 1+0.6\iu, \kappa_2 = 1-0.8\iu$}
		\end{subfigure}
	\end{center}
	\caption{Traces of the phase factor $e^{-\iu\phi_j}$ of the eigenmodes in the complex plane, demonstrating the non-Hermitian band inversion with low gain/loss (corresponding to unbroken $\P\T$ symmetry in \Cref{fig:inv0sym}).} \label{fig:invsym}
\end{figure}
	\begin{figure}[h]	
	\begin{center}
		\begin{subfigure}[b]{0.3\linewidth} 
			\centering
			\includegraphics[width=0.9\linewidth]{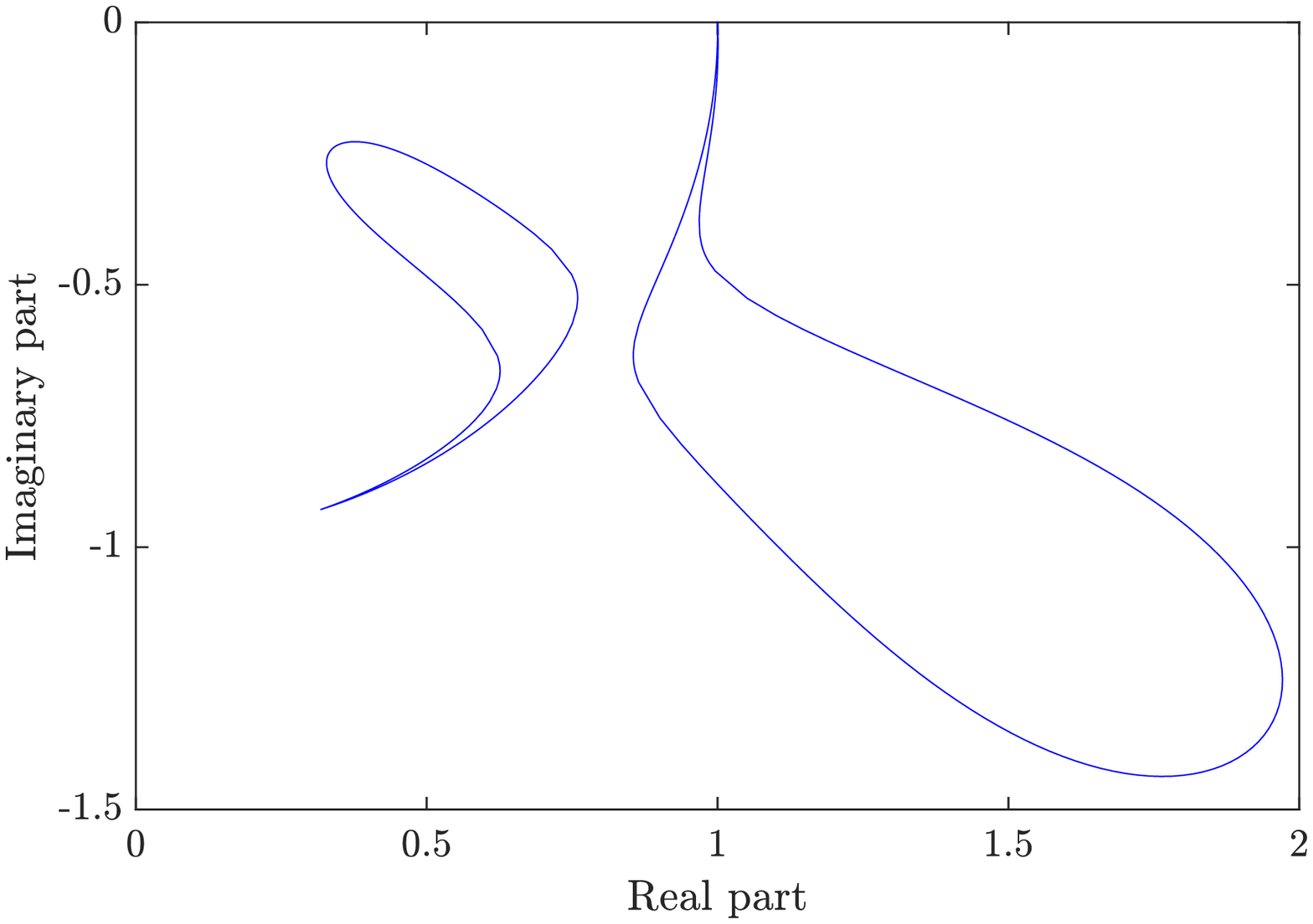}
			\caption{$\kappa_1 = 1+1.38\iu, \kappa_2 = 1-1.42\iu$}
		\end{subfigure}
		\hspace{0.1cm}
		\begin{subfigure}[b]{0.3\linewidth}
			\centering
			\includegraphics[width=0.9\linewidth]{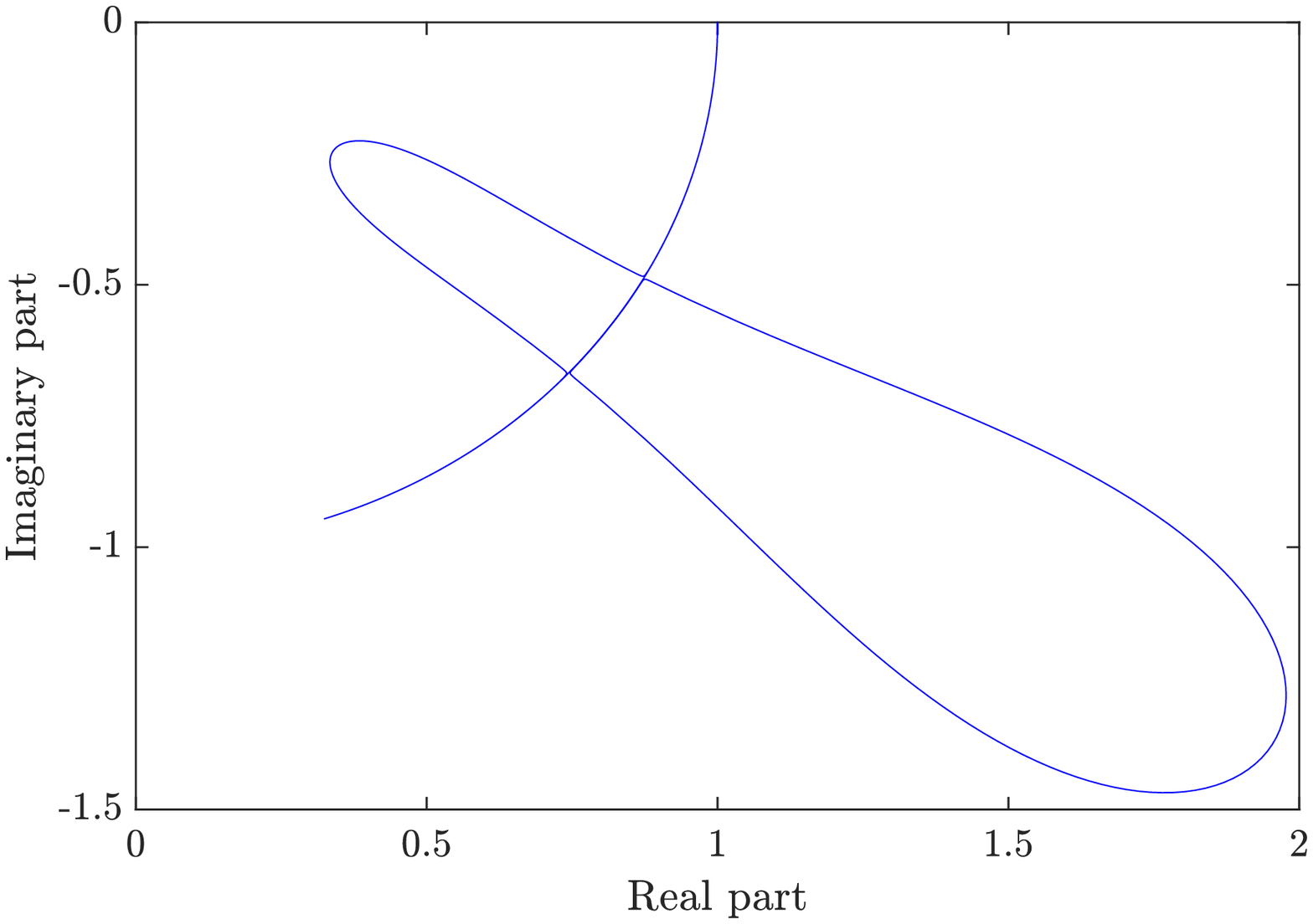}
			\caption{$\kappa_1 = 1+1.4\iu, \kappa_2 = 1-1.4\iu$} \label{fig:inv0}
		\end{subfigure}
		\hspace{0.1cm}
		\begin{subfigure}[b]{0.3\linewidth}
			\centering
			\includegraphics[width=0.9\linewidth]{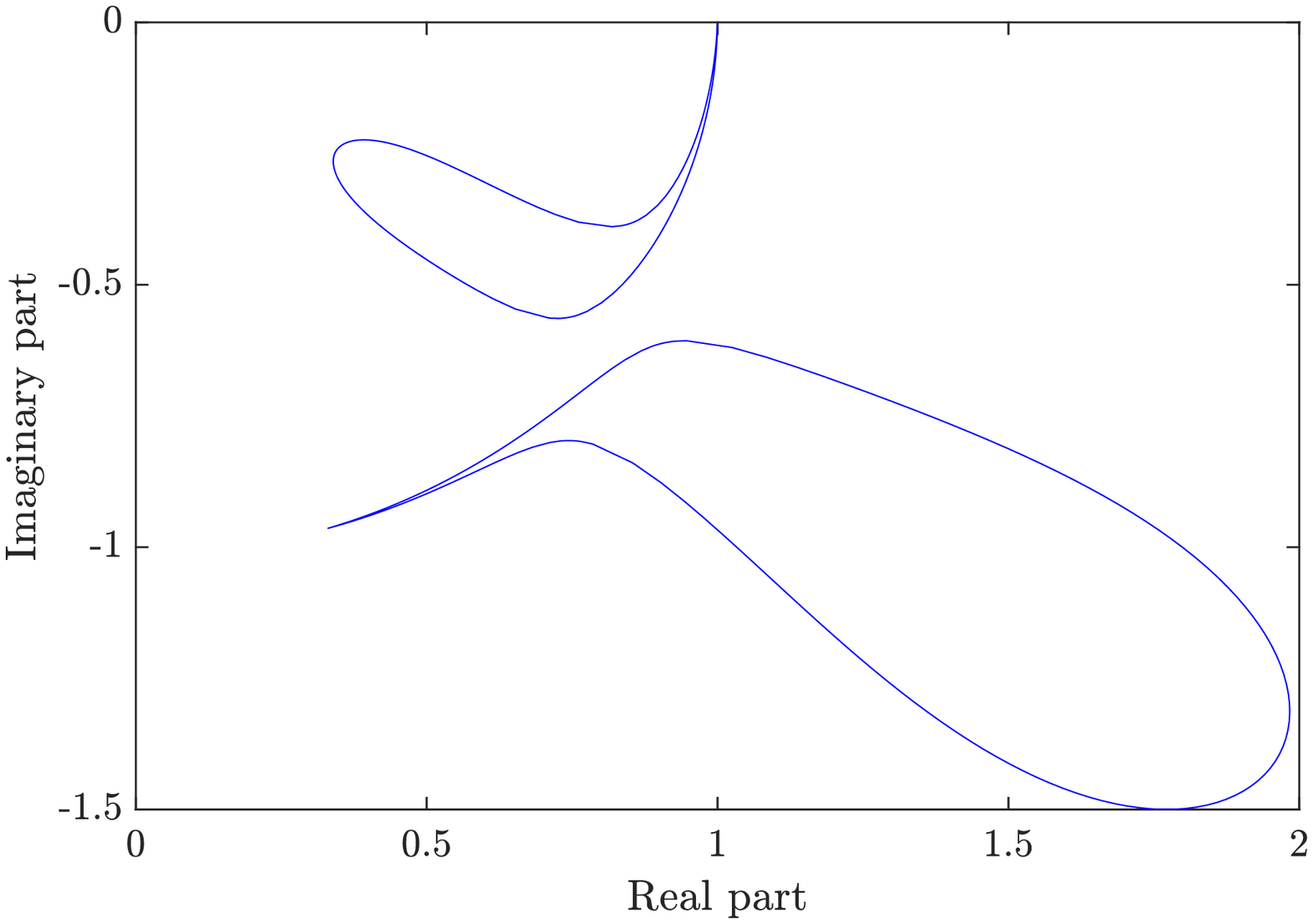}
			\caption{$\kappa_1 = 1+1.42\iu, \kappa_2 = 1-1.38\iu$}
		\end{subfigure}
	\end{center}
\caption{Traces of the phase factor $e^{-\iu\phi_j}$ of the eigenmodes in the complex plane, demonstrating the non-Hermitian band inversion high low gain/loss (corresponding to broken $\P\T$ symmetry in \Cref{fig:inv0}).}\label{fig:inv}
\end{figure}

	%\section{Edge-modes in active subwavelength metamaterials with a defect}
	\section{Localized modes by material-parameter defects} \label{sec:first}
	In this section, we study edge-modes in active subwavelength metamaterials with a defect in the gain/loss parameter but with periodic geometry.
	
	\subsection{Preliminary lemma}
	We begin the analysis with a general lemma that describes the subwavelength eigenmodes of a subwavelength resonator. We consider a single resonator $\Omega$ which is a connected domain $\Omega \in \R^3$ such that $\p \Omega$ is of Hölder class $C^{1,s}$ for some $0<s<1$. We denote the bulk modulus and density in $\Omega$ by $\kappa_b$, $\rho_b$ and assume there is a neighbourhood $U\subset \R^3$, $\Omega\subset U$ with material parameters $\kappa, \rho$. We introduce the parameters
	$$v_b = \sqrt{\frac{\kappa_b}{\rho_b}}, \quad v = \sqrt{\frac{\kappa}{\rho}}, \quad \delta = \frac{\rho_b}{\rho}, \quad k = \frac{\omega}{v}, \quad k_b = \frac{\omega}{v_b}.$$	
	We study the solutions to the problem
	\begin{equation} \label{eq:local_PDE}
	\left\{
	\begin{array} {ll}
	\ds \Delta {u}+ k^2 {u}  = 0 & \text{in } U \setminus \Omega, \\[0.3em]
	\ds \Delta {u}+ k_b^2 {u}  = 0 & \text{in } \Omega, \\
	\nm
	\ds  {u}|_{+} -{u}|_{-}  = 0  & \text{on } \partial \Omega, \\
	\nm
	\ds  \delta \frac{\partial {u}}{\partial \nu} \bigg|_{+} - \frac{\partial {u}}{\partial \nu} \bigg|_{-} = 0 & \text{on } \partial \Omega.
	\end{array}
	\right.
	\end{equation}
	The following result shows a fundamental property of high-contrast subwavelength resonators. Intuitively, the idea is that as $\delta \rightarrow 0$, the limiting problem is a homogeneous Neumann problem inside $\Omega$.
	\begin{lem}\label{lem:constant}
		As $\delta \rightarrow 0$, any solution $u$ to \eqref{eq:local_PDE} with $\omega = O(\sqrt{\delta})$ satisfies
		$$u(x) = u_\Omega\big(1 + O(\delta)\big), \quad x\in \Omega,$$
		for some constant $u_\Omega$. 
	\end{lem}
	\begin{proof}
		The solution $u$ can be represented as
		$$u(x) = \begin{cases} 
		\S_\Omega^{k_b}[\phi^\mathrm{in}](x), \quad &x \in \Omega,  \\ H(x) + \S_\Omega^{k}[\phi^\mathrm{out}](x), \quad &x \in U\setminus \overline \Omega, \end{cases}$$
		for some function $H$ satisfying $\Delta H + k^2 H = 0$ in  $U$. Clearly, if $u(x)$ is a solution, then $cu(x)$ is also a solution for any $c\in \mathbb{C}$, so we can assume that
		$$\|\phi^\mathrm{in}\|_{L^2(\p v)} = O(1), \quad  \|\phi^\mathrm{out}\|_{L^2(\p \Omega)} = O(1), \quad \|H\|_{H^1(\p \Omega)} = O(1)$$
		as $\delta \rightarrow 0$. From the boundary conditions, and using the jump relations \eqref{eq:jump1} and \eqref{eq:jump2}, we have 
		$$\left(-\frac{1}{2}I + \K_\Omega^{k_b,*}\right)[\phi^\mathrm{in}] = O(\delta).$$
		It then follows that $\phi^\mathrm{in} = \psi + O(\delta)$ for some $\psi \in \ker\left(-\frac{1}{2}I + \K_\Omega^{0,*}\right)$. It is well-known that $\S_\Omega^0[\psi](x)$ is constant for $x\in \Omega$ \cite{MaCMiPaP}. Moreover, from the low-frequency expansion \eqref{eq:Sexp} we have
		$$\S_\Omega^{k_b}[\phi^\mathrm{in}] = \S_\Omega^0[\psi] + O(\delta),$$
		which proves the claim.
	\end{proof}

		 	\begin{figure}[tbh]
		\centering
		\begin{tikzpicture}[scale=1.5]
		\pgfmathsetmacro{\rb}{0.25pt}
		\pgfmathsetmacro{\rs}{0.2pt}
		\coordinate (a) at (0.25,0);		
		\coordinate (b) at (1.05,0);	
		
		\draw (-0.5,0.85) -- (-0.5,-1);
		\def\Done{ plot [smooth cycle] coordinates {($(a)+(210:\rb)$) ($(a)+(270:\rs)$) ($(a)+(330:\rb)$) ($(a)+(30:\rs)$) ($(a)+(90:\rb)$) ($(a)+(150:\rs)$) 
		}};
		\def\Dtwo{ plot [smooth cycle] coordinates {($(b)+(30:\rb)$) ($(b)+(90:\rs)$) ($(b)+(150:\rb)$) ($(b)+(210:\rs)$) ($(b)+(270:\rb)$) ($(b)+(330:\rs)$) }};
		\draw \Done;
		\draw \Dtwo;
		\pattern[pattern=north east lines, opacity=0.3] \Done;
		\pattern[pattern=crosshatch dots, opacity=0.3] \Dtwo;
		\draw (a) node{$\kappa_2$};
		\draw (b) node{$\kappa_1$};
		\draw (0.65,0.9) node{$m=1$};

		\begin{scope}[xshift=-2.3cm]
		\coordinate (a) at (0.25,0);		
		\coordinate (b) at (1.05,0);	
		\draw[dashed, opacity=0.5] (-0.5,0.85) -- (-0.5,-1);		
		\draw \Done;
		\draw \Dtwo;
		\pattern[pattern=north east lines, opacity=0.3] \Dtwo;
		\pattern[pattern=crosshatch dots, opacity=0.3] \Done;
		\draw (a) node{$\kappa_1$};
		\draw (b) node{$\kappa_2$};
		\draw (0.65,0.9) node{$m=0$};
		\end{scope}
		\begin{scope}[xshift=-4.6cm]
		\coordinate (a) at (0.25,0);		
		\coordinate (b) at (1.05,0);	
		\draw \Done;
		\draw \Dtwo;
		\pattern[pattern=north east lines, opacity=0.3] \Dtwo;
		\pattern[pattern=crosshatch dots, opacity=0.3] \Done;
		\draw (a) node{$\kappa_1$};
		\draw (b) node{$\kappa_2$};
		\draw (0.65,0.9) node{$m=-1$};
		\begin{scope}[xshift = 1.2cm]
		\draw (-1.6,0) node{$\cdots$};
		\end{scope};
		\end{scope}
		\begin{scope}[xshift=2.3cm]
		\coordinate (a) at (0.25,0);		
		\coordinate (b) at (1.05,0);	
		\draw[dashed, opacity=0.5] (-0.5,0.85) -- (-0.5,-1);
		\draw \Done;
		\draw \Dtwo;
		\pattern[pattern=north east lines, opacity=0.3] \Done;
		\pattern[pattern=crosshatch dots, opacity=0.3] \Dtwo;
		\draw (a) node{$\kappa_2$};
		\draw (b) node{$\kappa_1$};
		\draw (0.65,0.9) node{$m=2$};
		\draw (1.7,0) node{$\cdots$};
		\end{scope}
		\end{tikzpicture}
		\caption[Single]{Illustration of the edge. The special case $\kappa_1 = \overline{\kappa_2}$ corresponds to local $\P \T$ symmetry. Legend:
			\raisebox{-2pt}{\tikz{
					\draw[pattern=crosshatch dots, opacity=0.3] (0,0) circle (5pt); 
					\draw (0,0) circle (5pt);}} bulk modulus $\kappa_1$,
			\raisebox{-2pt}{\tikz{
					\draw[pattern=north east lines, opacity=0.3] (0,0) circle (5pt); 
					\draw (0,0) circle (5pt);}} bulk modulus $\kappa_2$. \label{fig:edge} }
	\end{figure}	

	\subsection{Localized modes}
	We will begin this section by deriving an asymptotic eigenvalue problem that characterises localized modes in the structure defined in \Cref{sec:prelim} with a general distribution of the bulk moduli. Then, we will compute an asymptotic formula for resonant frequencies of localized modes, and corresponding decay lengths, of the defect structure illustrated in \Cref{fig:edge}.
	
	Assume that $u$ is a simple localized eigenmode to \eqref{eq:scattering} in the subwavelength regime, \ie{}, $u$ corresponds to a simple eigenvalue $\omega$ which scales as $O(\delta^{1/2})$. Here, we refer to localization in the $L^2(\R)$-sense, \ie{} $\int_{\R}|u(x_1,x_2,x_3)|^2 \dx x_1 < \infty$ for all $x_2, x_3$. In addition, we assume that $u$ is normalized as $\int_{\R}|u(x_1,0,0)|^2 \dx x_1 = 1$. By \Cref{lem:constant}, we have
	$$u(x) = u_i^m +O(\delta), \quad x\in D_i^m,$$
	for some constant values $u_i^m, i=1,2, m\in \Z$.
	
	Since $\omega$ is subwavelength, we can write
	\begin{equation}\label{eq:omega}
	\omega = \omega_0 + O(\delta), \quad \omega_0 = \beta \delta^{1/2},
	\end{equation}
	for some constant $\beta$. The first proposition holds for general values of $\kappa_i^m$.
	\begin{prop} \label{prop:loc_main}
	Any localized solution $u$ to \eqref{eq:scattering}, corresponding to a subwavelength frequency $\omega$, satisfies the equation
	\begin{equation} \label{eq:main}
	\frac{1}{\rho}\begin{pmatrix} C_{11}^\alpha & C_{12}^\alpha \\[0.3em] C_{21}^\alpha & C_{22}^\alpha\end{pmatrix} \begin{pmatrix}\ds \sum_{m\in \Z} u_1^me^{\iu \alpha mL} \\[0.3em] \ds \sum_{m\in \Z} u_2^me^{\iu \alpha mL} \end{pmatrix} = \mu\begin{pmatrix} \ds \sum_{m\in \Z} \frac{u_1^me^{\iu \alpha mL}}{\kappa_1^m}  \\[0.3em]  \ds \sum_{m\in \Z} \frac{u_2^me^{\iu \alpha mL}}{\kappa_2^m} \end{pmatrix}, \qquad \mu = \omega_0^2|D_1|.
	\end{equation}
	\end{prop}
	\begin{proof}	
	Taking the Floquet transform, we find from \eqref{eq:scattering} that 
	\begin{equation} \label{eq:scattering_quasi}
	\left\{
	\begin{array} {ll}
	\ds \Delta {u^\alpha}+ k^2 {u^\alpha}  = 0 & \text{in } Y \setminus D, \\
	\nm
	\ds \Delta {u^\alpha}+ \omega^2\rho_b\sum_{m\in \Z} \frac{e^{\iu \alpha mL}}{\kappa_i^m}u(x+mL)  = 0 & \text{in } D_i, \\
	\nm
	\ds  {u^\alpha}|_{+} -{u^\alpha}|_{-}  =0  & \text{on } \p D, \\
	\nm
	\ds  \delta \frac{\partial {u^\alpha}}{\partial \nu} \bigg|_{+} - \frac{\partial {u^\alpha}}{\partial \nu} \bigg|_{-} =0 & \text{on } \p D, \\
	\nm
	\ds u^\alpha(x+mL\w)= e^{\iu \alpha m} u^\alpha(x) & \text{for all } m \in \Z, \\
	\nm
	\ds u^\alpha(x_1,x_2,x_3)& \text{satisfies the $\alpha$-quasiperiodic outgoing radiation condition} \\ &\hspace{0.5cm} \text{as } \sqrt{x_2^2+x_3^2} \rightarrow \infty,
	\end{array}
	\right.
	\end{equation}
	where $u^\alpha(x) = \F[u](x,\alpha)$. Moreover inside $D_i$ we have, from \Cref{lem:constant}, that
	$$u^\alpha(x) = u_i^\alpha + O(\delta), \quad x\in D_i, \qquad u_i^\alpha = \sum_{m\in \Z}u_i^me^{\iu \alpha m L},$$
	for some sequences $u_i^m\in \ell^ 2(\Cb)$ for $i=1,2$. Observe, in particular, that $u_i^\alpha$ are constant in $x$. Following the arguments in \cite[Lemma 4.2]{ammari2020exceptional}, it then follows that 
	$$u^\alpha(x) = u_1^\alpha V_1^\alpha(x) + u_2^\alpha V_2^\alpha(x) + O(\delta^{1/2}), \quad x\in Y\setminus D.$$
	On one hand, using the transmission conditions and integration by parts, we obtain
	$$\int_{\p D_i} \frac{\p u^\alpha}{\p \nu}\bigg|_+ \dx \sigma = \frac{1}{\delta} \int_{\p D_i} \frac{\p u^\alpha}{\p \nu}\bigg|_- \dx \sigma = -\frac{\omega^2\rho_b}{\delta} \int_{D_i} \sum_{m\in \Z} \frac{e^{\iu \alpha mL}}{\kappa_i^m}u^\alpha(x) \dx x = -\frac{\omega^2\rho_b|D_i|}{\delta} \sum_{m\in \Z} \frac{u_i^m e^{\iu \alpha mL}}{\kappa_i^m} + O(\delta).$$
	On the other hand, we have
	$$\int_{\p D_i} \frac{\p u^\alpha}{\p \nu}\bigg|_+ \dx \sigma = u_1^\alpha \int_{\p D_i} \frac{\p V_1^\alpha}{\p \nu} \dx \sigma  + u_2^\alpha \int_{\p D_i} \frac{\p V_2^\alpha}{\p \nu} \dx \sigma  + O(\delta^{1/2}) =  -u_1^\alpha C_{i,1} -  u_2^\alpha C_{i,2} + O(\delta^{1/2}).$$
	Combining the above estimates, and using \eqref{eq:omega}, proves the result.
	\end{proof}
	
	Next, we will consider a structure with a defect as illustrated in \Cref{fig:edge}. For $\kappa_1, \kappa_2 \in \mathbb{C}$, we set 
	\begin{equation}\label{eq:k}
	\kappa_1^m = \begin{cases}\kappa_1, & m \leq 0, \\ \kappa_2, &m > 0,\end{cases} \qquad \kappa_2^m = \begin{cases}\kappa_2, & m \leq 0, \\ \kappa_1, & m > 0.\end{cases}
	\end{equation}
	Observe that under this assumption, the total structure is $\P$-symmetric, and consists of two half-space arrays as studied in \Cref{sec:periodic}. The intuition for studying this structure comes from \Cref{prop:zak}: in the general case, the Zak phases of the two periodic arrays, corresponding to the half-space arrays, have opposite sign.
	
	The special case $\kappa_1 = \overline{\kappa_2} := \kappa$ corresponds to a micro-structure that is $\P \T$-symmetric, \ie{} the unit cells of the half-space arrays are $\P\T$-symmetric. As we shall see, the behaviour is different in the case of unbroken $\P\T$ symmetry ($\kappa_1 = \overline{\kappa_2}$ with small $\mathrm{Im}(\kappa)$) compared to the case of broken $\P\T$-symmetry ($\kappa_1 = \overline{\kappa_2}$ with large $\mathrm{Im}(\kappa)$) or without $\P\T$ symmetry ($\kappa_1 \neq \overline{\kappa_2}$). In the case of unbroken $\P\T$ symmetry, the Zak phase vanishes to leading order, and we shall see that there are no localized modes in this case.
	
	We define
	$$U_1 = \sum_{m\leq 0} u_1^m e^{\iu \alpha mL}, \quad U_2 = \sum_{m>0} u_1^m e^{\iu \alpha mL}, \quad U_3 = \sum_{m\leq 0} u_2^m e^{\iu \alpha mL}, \qquad U_4 = \sum_{m>0} u_2^m e^{\iu \alpha mL}.$$
	\begin{lem} \label{lem:b}
		We have 
		$$ U_1 = bU_3, \quad U_4 = bU_2, $$
		for some $b\in \mathbb{C}$, independent of $\alpha$, satisfying $|b| < 1$.
	\end{lem}
	\begin{proof}
		Observe first that, due to the $\P$-symmetry of the structure, we have $ U_1 = bU_3$ and  $U_4 = bU_2$ for some $b =b(\alpha)$ which might depend on $\alpha$. $b$ being constant in $\alpha$ is equivalent to 
		$$\frac{u_1^n}{u_2^n}$$
		being independent of $n$ for $n\leq 0$. To prove this, we will apply the inverse Floquet transform to the equation in \Cref{prop:loc_main}. We first define 
		$$C_{ij}^m = \frac{L}{2\pi} \int_{Y} C_{ij}^\alpha e^{\iu \alpha m L} \dx \alpha, \qquad m \in \Z.$$
		We then find from \eqref{eq:main} that
		\begin{align*}
		\frac{1}{\rho} \sum_{m \in \Z} C_{11}^{m-n}u_1^m + C_{12}^{m-n} u_2^m &= \frac{\mu}{\kappa_1^m} u_1^n , \\
		\frac{1}{\rho} \sum_{m \in \Z} C_{21}^{m-n}u_1^m + C_{22}^{m-n} u_2^m &= \frac{ \mu}{\kappa_2^m} u_2^n,
		\end{align*}
		for $n \in \Z$. We define 
		$$U^m = \begin{pmatrix} u_1^m \\ u_2^m\end{pmatrix}, \quad K = \begin{pmatrix} \kappa_1 & 0 \\ 0 & \kappa_2 \end{pmatrix}, \quad J = \begin{pmatrix} 0 & 1 \\ 1 & 0 \end{pmatrix}, \quad C^m= \begin{pmatrix} C_ {11}^m & C_ {12}^m \\ C_ {21}^m & C_ {22}^m \end{pmatrix}, $$
		and introduce the doubly infinite vectors and matrices
		$$\uf = \left(\begin{smallmatrix} \svdots \\ U^{-1} \\ U^0\\U^1\\U^2\\ \svdots\end{smallmatrix}\right),\quad \K = \left(\begin{smallmatrix} \sddots & \svdots & \svdots  & \svdots & \svdots & \sadots  \\
		\cdots & K & 0 & 0 & 0 & \cdots \\
		\cdots & 0 & K & 0 & 0 & \cdots \\
		\cdots & 0 & 0 & JK & 0 & \cdots \\
		\cdots & 0 & 0 & 0 & JK & \cdots \\
		\sadots & \svdots & \svdots  & \svdots & \svdots & \sddots \end{smallmatrix}\right), \quad \C = \left(\begin{smallmatrix} \sddots & \svdots & \svdots  & \svdots & \svdots & \sadots  \\
		\cdots & C^0 & C^1 & C^2& C^3  & \cdots \\
		\cdots & C^{-1} & C^0 & C^{1}& C^2  & \cdots \\
		\cdots & C^{-2} & C^{-1} & C^0 & C^1 & \cdots \\
		\cdots & C^{-3} & C^{-2} & C^{-1} & C^0 & \cdots \\
		\sadots & \svdots & \svdots  & \svdots & \svdots & \sddots \end{smallmatrix}\right).$$
		Observe that $\C$ is the Laurent operator whose symbol is the quasiperiodic capacitance matrix $C^\alpha$. Here, $\K$ and $\C$ are operators on $\ell_2^2(\Cb)$, which is defined as the space of square-summable sequences of vectors in $\Cb^2$. We then find that $(\rho\mu,\uf)$ is an eigenpair solution to the spectral problem
		$$\K \C \uf= \rho \mu \uf.$$
		Moreover, we see that $\tilde{\uf} := \left(\begin{smallmatrix} \cdots & U^{-2} & U^{-1}&U^2&U^3& \cdots\end{smallmatrix}\right)^\mathrm{T}$, where the superscript $\mathrm{T}$ denotes the transpose,  is an eigenvector corresponding to the same eigenvalue $\rho \mu$. Since $\rho \mu$ is a simple eigenvalue, it follows that
		$$\tilde\uf =  a \uf$$
		for some $a\in \Cb$. In other words, we have that $u_1^n/ u_2^n$ is constant for all  $n \leq 0$ and that $b= u_1^n/ u_2^n$ is constant in $\alpha$. Moreover, since $\uf$ is square-summable, it follows that $b < 1$ which proves the claim.
	\end{proof}
	The constant $b$ can be interpreted as the decay of the localized mode between two resonators. Using \Cref{lem:b}, we have 
	$$ \sum_{m\in \Z} u_1^me^{\iu \alpha mL} = U_2 + bU_3, \qquad \sum_{m\in \Z} u_2^me^{\iu \alpha mL} = bU_2 + U_3.$$
	Moreover, using \eqref{eq:k} we obtain 
	$$\sum_{m\in \Z} \frac{u_1^me^{\iu \alpha mL}}{\kappa_1^m} = \frac{U_2}{\kappa_2} + \frac{bU_3}{\kappa_1} , \qquad \sum_{m\in \Z} \frac{u_2^me^{\iu \alpha mL}}{\kappa_2^m} = \frac{bU_2}{\kappa_1} + \frac{U_3}{\kappa_2}.$$
	In total, \eqref{eq:main} reads
	$$
	\frac{1}{\rho}\begin{pmatrix} C_{11}^\alpha & C_{12}^\alpha \\[0.3em] C_{21}^\alpha & C_{22}^\alpha\end{pmatrix} \begin{pmatrix} 1 & b \\[0.3em] b & 1\end{pmatrix} \begin{pmatrix}U_2 \\[0.3em] U_3 \end{pmatrix} = \mu\begin{pmatrix} \ds \frac{1}{\kappa_2} & \ds \frac{b}{\kappa_1} \\[0.3em] \ds \frac{b}{\kappa_1} & \ds \frac{1}{\kappa_2}\end{pmatrix} \begin{pmatrix}U_2 \\[0.3em] U_3 \end{pmatrix},
	$$
	or, by defining
	$$A = \begin{pmatrix} 1 & b \\[0.3em] b & 1\end{pmatrix}, \quad B = \rho\begin{pmatrix} \ds \frac{1}{\kappa_2} & \ds \frac{b}{\kappa_1} \\[0.3em] \ds \frac{b}{\kappa_1} & \ds \frac{1}{\kappa_2}\end{pmatrix}, \quad v = \begin{pmatrix}U_2 \\[0.3em] U_3 \end{pmatrix},$$
	we arrive at
	$$B^{-1}C^\alpha Av = \mu v.$$
	We have therefore proven the following result.
	\begin{prop}
		Assume the array of resonators has a topological defect specified by \eqref{eq:k}. Then there is a localized mode in the subwavelength regime, corresponding to a simple eigenvalue frequency $\omega$, only if there is a $\mu\in \Cb$, independent of $\alpha$, such that $\mu$ is an eigenvalue of $B^{-1}C^\alpha A$ for all $\alpha \in Y^*$.
	\end{prop}
	
	The eigenvalues of $B^{-1}C^\alpha A$ are given by
	\begin{align*}
	\mu_j^\alpha(b) &= \frac{\kappa_1\kappa_2}{\rho(b^2\kappa_2^2 - \kappa_1^2) }\Bigg(C_{11}^\alpha\left(b^2\kappa_2-\kappa_1\right) + b\left(\kappa_2-\kappa_1\right)\Re(C_{12}^\alpha) \\
	&\qquad \qquad + (-1)^j\sqrt{\big(C_{11}^\alpha\left(b^2\kappa_2-\kappa_1\right) + b\left(\kappa_2-\kappa_1\right)\Re(C_{12}^\alpha)\big)^2-(b^2-1)(b^2\kappa_2^2-\kappa_1^2)\big((C_{11}^\alpha)^2 - |C_{12}^\alpha|^2\big)}\Bigg).
	\end{align*}
	Here, we choose a holomorphic branch of the square root around a neighbourhood of the curve $f(\alpha)$ for $\alpha \in (-\pi/L,\pi/L]$, where
	$$f(\alpha) = \big(C_{11}^\alpha\left(b^2\kappa_2-\kappa_1\right) + b\left(\kappa_2-\kappa_1\right)\Re(C_{12}^\alpha)\big)^2-(b^2-1)(b^2\kappa_2^2-\kappa_1^2)\big((C_{11}^\alpha)^2 - |C_{12}^\alpha|^2\big).$$
	We also define the curve $g(\alpha)$ as
	$$g(\alpha) = \left(C_{12}^\alpha\left(b^2\kappa_2-\kappa_1\right) + C_{11}^\alpha b\left(\kappa_2-\kappa_1\right)\right)^2.$$
	At points $\alpha$ where $C_{12}^\alpha$ is real, we have $f(\alpha) = g(\alpha)$, and therefore $\mu_j^\alpha = \mu_\pm^\alpha$,  where
	\begin{align*}
	\mu_\pm^\alpha= \frac{\kappa_1\kappa_2(b\pm1)}{\rho(b\kappa_2 \pm \kappa_1)}(C_{11}^\alpha \pm C_{12}^\alpha).
	\end{align*}
	In particular, the above formula holds at $\alpha = 0$ and at $\alpha = \pi/L$. Here, the sign depends on the branch of the square root. Since the capacitance coefficients $C_{ij}^\alpha$ depend on $\alpha$, the eigenvalue $\mu_j^\alpha$ can only be constant in $\alpha$ if the points $\alpha = 0$ and $\alpha = \pi/L$ correspond to opposite signs. In other words, the concatenation $\gamma$ of the curves defined by $f$ and by $g$ must have odd winding number around the origin, resulting in a swap of branches.
	
	At $\alpha = 0$ we have $C_{12}^0 = - C_{11}^0$, and so
	\begin{align*}
	\mu_-^0 = 
	\ds 2C_{11}^0\frac{\kappa_1\kappa_2(b-1)}{\rho(b\kappa_2 - \kappa_1)},\qquad \mu_+^0 = 0.
	\end{align*}
	Therefore, we always have $\mu_+^0 \neq \mu_-^{\pi/L}$. The condition $\mu_-^0 = \mu_+^{\pi/L}$ is equivalent to 
	$$\lambda_2 \frac{b-1}{b\kappa_2 - \kappa_1} = \lambda_1\frac{b+1}{b\kappa_2 + \kappa_1}.$$
	Here, $\lambda_2 = 2C_{11}^0$ and $\lambda_1 = C_{11}^{\pi/L} + C_{12}^{\pi/L}$, which was proved in \cite{ammari2020robust} to satisfy $0<\lambda_1 < \lambda_2$. From this we find two possible values for $b$:
	\begin{equation} \label{eq:b}
	b_\pm = \frac{1}{2}\left(l\left(1 - \frac{\kappa_1}{\kappa_2}\right) \pm \sqrt{l^2\left(1 - \frac{\kappa_1}{\kappa_2}\right)^2 + \frac{4\kappa_1}{\kappa_2}}\right), \qquad l = \frac{\lambda_2+\lambda_1}{\lambda_2-\lambda_1}.
	\end{equation}
	In the case when the micro-structure is $\P\T$-symmetric with unbroken $\P\T$-symmetry, \ie{}, if $\kappa_1 = \overline{\kappa_2} := \kappa$, with 
	$$|\mathrm{Im}(\kappa)| \leq \frac{\mathrm{Re}(\kappa)}{\sqrt{l^2-1}},$$
	the two values of $b$ have unit modulus: $|b_-| = |b_+| = 1$. In the case of broken $\P\T$-symmetry, \ie{}, $|\mathrm{Im}(\kappa)| > \frac{\mathrm{Re}(\kappa)}{\sqrt{l^2-1}}$, or in the case $\kappa_1 \neq \overline{\kappa_2}$, there will always be a value $b_0$ with $|b_0| < 1$ and a value $b_1$ with $|b_1| > 1$.
	
	We have now proven the following main result.
	\begin{thm} \label{thm:main}
		Assume the array of resonators has a defect in the material parameters specified by \eqref{eq:k}. Then,
		\begin{itemize}
			\item if $\kappa_1 = \overline{\kappa_2} := \kappa$ with $|\mathrm{Im}(\kappa)| \leq \frac{\mathrm{Re}(\kappa)}{\sqrt{l^2-1}}$ (unbroken $\P\T$-symmetry), the structure does not support simple localized modes in the subwavelength regime.
			\item if $\kappa_1 = \overline{\kappa_2} := \kappa$ with $|\mathrm{Im}(\kappa)| > \frac{\mathrm{Re}(\kappa)}{\sqrt{l^2-1}}$ (broken $\P\T$-symmetry) or if $\kappa_1 \neq \overline{\kappa_2}$ (no $\P\T$-symmetry), the frequency $\omega$ of a simple localized mode in the subwavelength regime must satisfy
			$$\omega = \sqrt{\frac{\mu_j^\alpha(b_0)}{|D_1|}} + O(\delta).$$
			Here, $b_0$ is the value of $b$ specified by \eqref{eq:b} satisfying $|b_0| < 1$. 
		\end{itemize}
	\end{thm}
	\begin{rmk}
		The second part of \Cref{thm:main} describes the possible frequency $\omega$ and the decay length $b$ of simple localized modes in the subwavelength regime. However, we have not proved the existence of such modes. To do this, the main remaining challenge is to prove that one eigenvalue $\mu_j(b_0)$ indeed is constant in $\alpha$. Analytically, this is obscured by the fact that the capacitance coefficients have a complicated dependency on $\alpha$. Numerically, however, the eigenvalues $\mu$ are straight-forward to compute (shown in \Cref{fig:munpt}), showing one constant eigenvalue.
	\end{rmk}
	
	\subsection{Numerical illustrations}
	The resonant frequencies and corresponding eigenmodes were computed using the multipole method, in the case of circular resonators. In \Cref{fig:bandpt,fig:bandnpt}, the defect frequencies, computed as in \Cref{thm:main}, are shown in the complex plane. The two points correspond to the two distinct defect structures that can be made with given values of the bulk moduli. To verify the existence of the edge mode, the resonant modes of a finite but large system were computed. \Cref{fig:mode} shows the localized mode in a truncated array with 48 resonators, with a topological defect as depicted in \Cref{fig:edge}. \Cref{fig:munpt} shows the eigenvalues of the matrix $B^{-1}C^\alpha A$ as functions of $\alpha$, clearly illustrating that one eigenvalue is constant. For comparison, \Cref{fig:model} shows the localized mode for small $\kappa_1, \kappa_2$, demonstrating a significantly less localized mode. In both cases, the relative discrepancy $e_\omega$ of the frequencies computed using \Cref{thm:main} and using the multipole discretization is around $e_\omega \approx 0.1\%$.
	
	In the case of small gain and loss, the localized mode disappears in the limit of balanced gain and loss, \ie{} when $\mathrm{Im}(\kappa_1) \rightarrow \mathrm{Im}(\kappa_2)$. In this limit, the magnitude of $b$ approaches unity and the rate of decay of the localized mode approaches zero.
	
	\begin{figure}[h]
	\begin{minipage}[b]{0.49\linewidth}
		\centering
		\includegraphics[width=0.95\linewidth]{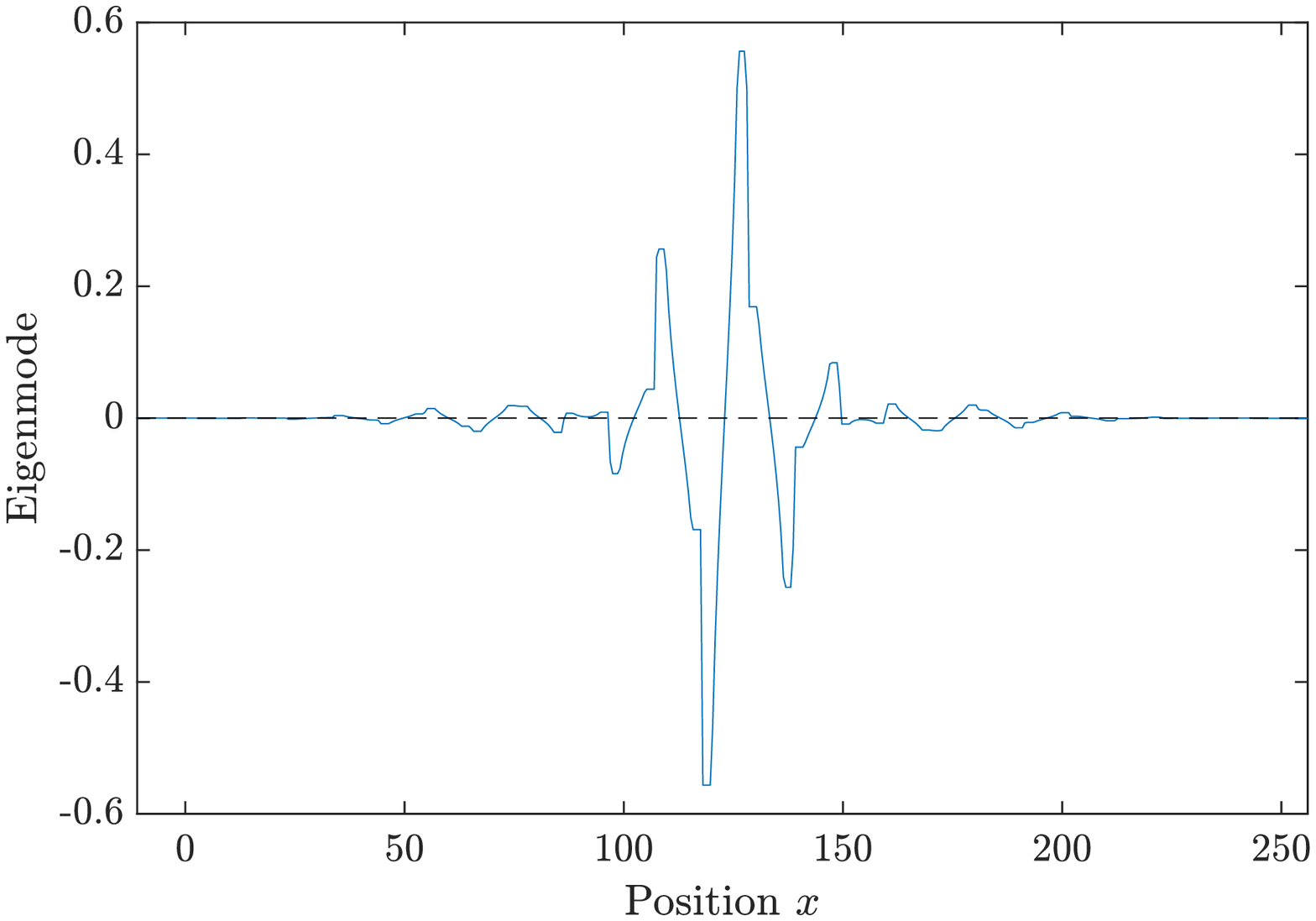}		
		\caption{Plot of the localized mode in a finite but large array of resonators, satisfying $|b| \approx 0.44$ and with relative error $e_\omega \approx 0.11\%$. Here, we use $\kappa_1 = 1+1.38\iu, \kappa_2 = 1-1.42\iu$. \\ \hspace{1pt}} \label{fig:mode}
	\end{minipage}
	\hspace{0.5cm}
	\begin{minipage}[b]{0.49\linewidth}
		\centering
		\includegraphics[width=0.95\linewidth]{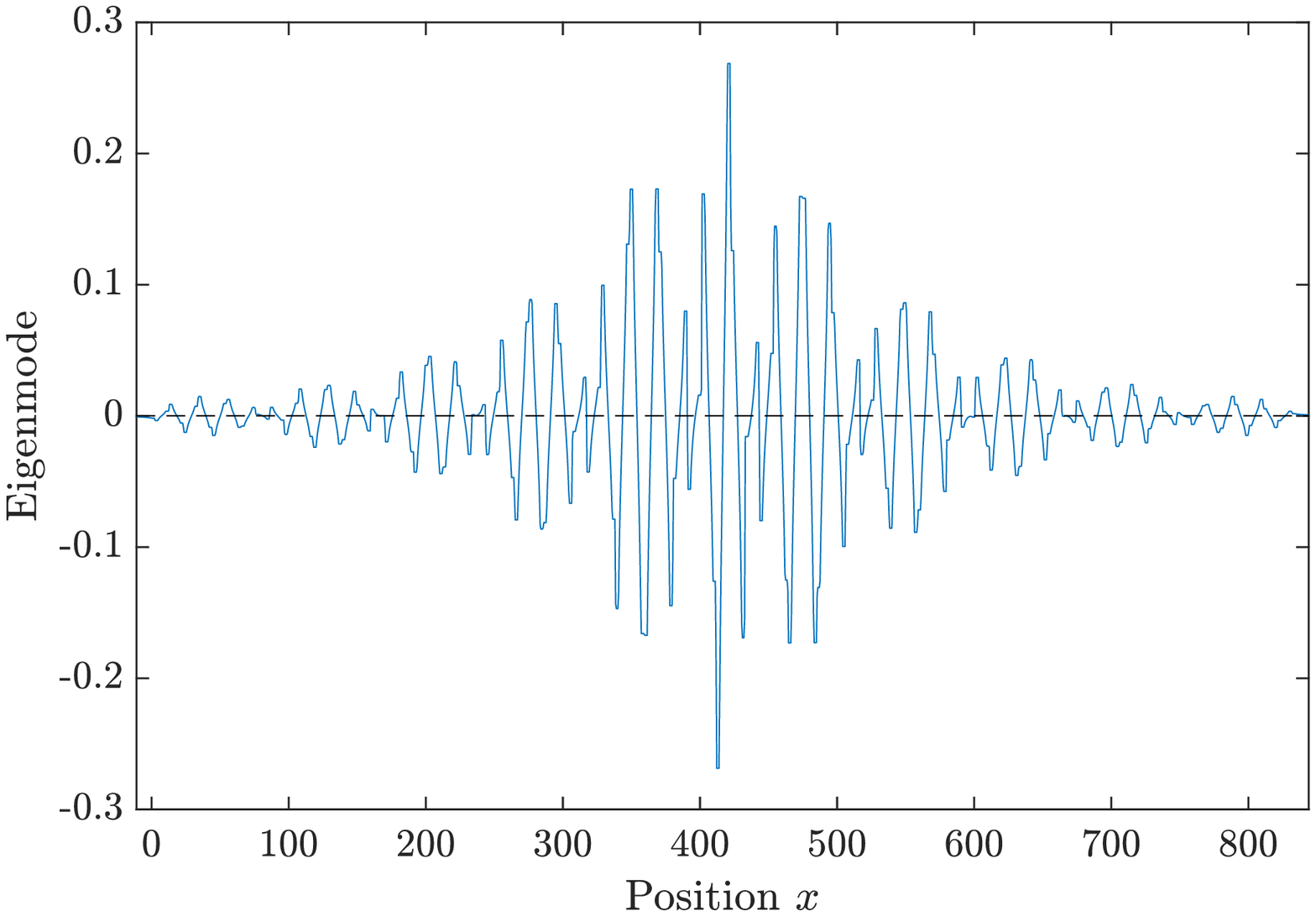}
	\caption{Plot of the localized mode in a finite but large array of resonators, satisfying $|b| \approx 0.88$ and with relative error $e_\omega \approx 0.09\%$. Observe the different $x$-axis scale compared to \Cref{fig:mode}. Here, we use $\kappa_1 = 1+0.8\iu, \kappa_2 = 1-0.6\iu$.} \label{fig:model}
	\end{minipage}
	\end{figure}

\begin{figure}[h]
	\centering
	\includegraphics[width=0.5\linewidth]{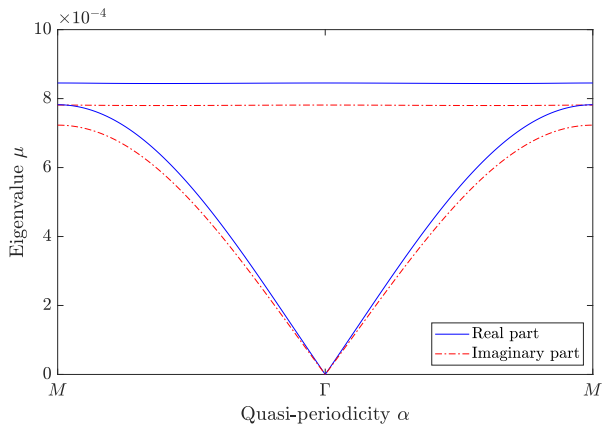}
	\caption{Real and imaginary parts of the eigenvalues $\mu$ of $B^{-1}C^\alpha A$, showing a flat band in $\alpha$. Here, we use $\kappa_1 = 1+1.38\iu, \kappa_2 = 1-1.42\iu$.} \label{fig:munpt}
\end{figure}
	
	\section{Localized modes by geometrical defects} \label{sec:geomd}
	In this section, we study a structure, again, created by joining two %\todo{revise}
	half-structures with different Zak phases. Here, in contrast to the structure studied in \Cref{sec:first}, the difference is achieved by introducing a geometrical defect, depicted in \Cref{fig:finite}. Similar structures have been extensively studied in the case of tight-binding Hamiltonian systems (see, for example, \cite{drouot1,fefferman,fefferman_mms}). It is worth emphasizing  that the current structure is a non-Hermitian generalization of the structure studied in \cite{ammari2019topologically}. Here, for completeness, we will demonstrate numerically the edge modes in the non-Hermitian case.
	
	\subsection{Problem statement}
	We consider a dimerized array with a single, centre resonator.
	We let $d>0$ and $d'>0$ be the resonator separations and assume $d< d'$. Then we assume that the truncated array $D$ can be written
	\begin{equation} \label{finite_form}
	D = \left(\bigcup_{n=-M}^{M} D_0 + n(d+d',0,0)  \right) \bigcup \left( \bigcup_{n=-M+1}^M D_0 + n(d+d',0,0) - (d',0,0)\right),
	\end{equation}
	where $D_0$ is the centre resonator. In other words, $D$ consists of a single centre resonator surrounded by pairs of resonators. Moreover, $N=4M+1$ is the number of resonators.

	As before, we assume that all the resonators have the same density $\rho_b \in \R$. Moreover, the bulk moduli are given by $\kappa, \overline{\kappa} $ and $\kappa_0$ where $\kappa \in \mathbb{C}$ and $\kappa_0 = \mathrm{Re}(\kappa)$. We assume that the centre resonator $D_0$ has bulk modulus $\kappa_0 $, and that the remaining resonators have bulk moduli
	$$\kappa \text{ in } \begin{cases} D_0 + n(d+d',0,0), & n < 0, \\ D_0 + n(d+d',0,0) - (d',0,0), & n \geq 1, \end{cases} \text{ and } \overline{\kappa} \text{ in } \begin{cases} D_0 + n(d+d',0,0), & n > 0, \\ D_0 + n(d+d',0,0) - (d',0,0), & n < 1. \end{cases}$$
	This distribution is illustrated in \Cref{fig:finite}. We also assume that $\mathrm{Im}(\kappa)$ is chosen below the exceptional point of the corresponding infinite structure, so that the Zak phase $\phi_j^{\mathrm{zak}}$ is well-defined. 
	
	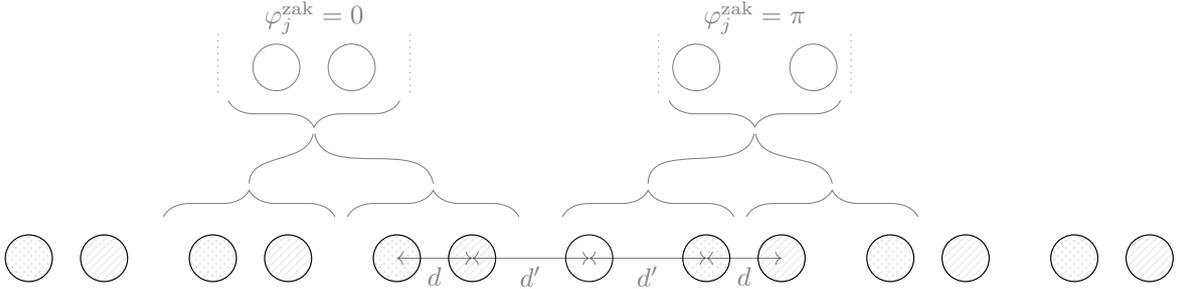
\begin{figure}[t]
		\centering
		\begin{tikzpicture}[scale=1.1]

		%\pattern[pattern=north east lines, opacity=0.3] \Done;
		%\pattern[pattern=crosshatch dots, opacity=0.3] \Dtwo;
		%\draw (a) node{$\kappa_2$};
		%\draw (b) node{$\kappa_1$};
		
		\begin{scope}
		\draw (0.65,0) coordinate (start1) circle (8pt);
		\draw[<->, opacity=0.5] (0.65,0) -- (2.05,0) node[pos=0.5, yshift=-7pt,]{$d'$};
		\draw[<->, opacity=0.5] (-0.75,0) -- (0.65,0) node[pos=0.5, yshift=-7pt,]{$d'$};
		\end{scope}
		
		\begin{scope}[xshift=-1.85cm]
		\draw[<->, opacity=0.5] (0.2,0) -- (1.1,0) node[pos=0.5, yshift=-7pt,]{$d$};
		\draw[pattern=crosshatch dots, opacity=0.3] (0.2,0) circle (8pt);
		\draw (0.2,0) circle (8pt);		
		\begin{scope}[xshift = 1.3cm]
		\draw[pattern=north east lines, opacity=0.3] (-0.2,0) circle (8pt);
		\draw (-0.2,0) circle (8pt);
		\end{scope};
		\end{scope};
		
		\begin{scope}[xshift=1.85cm]
		\draw[<->, opacity=0.5] (0.2,0) -- (1.1,0) node[pos=0.5, yshift=-7pt,]{$d$};
		\draw[pattern=crosshatch dots, opacity=0.3] (0.2,0) circle (8pt);
		\draw (0.2,0) circle (8pt);		
		\begin{scope}[xshift = 1.3cm]
		\draw[pattern=north east lines, opacity=0.3] (-0.2,0) circle (8pt);
		\draw (-0.2,0) circle (8pt);
		\end{scope};
		\end{scope};
		
		\begin{scope}[xshift=4.05cm]
		\draw[pattern=crosshatch dots, opacity=0.3] (0.2,0) circle (8pt);
		\draw (0.2,0) circle (8pt);		
		\begin{scope}[xshift = 1.3cm]
		\draw[pattern=north east lines, opacity=0.3] (-0.2,0) circle (8pt);
		\draw (-0.2,0) circle (8pt);
		\end{scope};
		\end{scope};
		
		\begin{scope}[xshift=6.25cm]
		\draw[pattern=crosshatch dots, opacity=0.3] (0.2,0) circle (8pt);
		\draw (0.2,0) circle (8pt);		
		\begin{scope}[xshift = 1.3cm]
		\draw[pattern=north east lines, opacity=0.3] (-0.2,0) circle (8pt);
		\draw (-0.2,0) circle (8pt);
		\end{scope};
		\end{scope};
		
		\begin{scope}[xshift=-4.05cm]
		\draw[pattern=crosshatch dots, opacity=0.3] (0.2,0) circle (8pt);
		\draw (0.2,0) circle (8pt);		
		\begin{scope}[xshift = 1.3cm]
		\draw[pattern=north east lines, opacity=0.3] (-0.2,0) circle 	(8pt);
		\draw (-0.2,0) circle (8pt);
		\end{scope};
		\end{scope};	
		
		\begin{scope}[xshift=-6.25cm]
		\draw[pattern=crosshatch dots, opacity=0.3] (0.2,0) circle (8pt);
		\draw (0.2,0) circle (8pt);		
		\begin{scope}[xshift = 1.3cm]
		\draw[pattern=north east lines, opacity=0.3] (-0.2,0) circle (8pt);
		\draw (-0.2,0) circle (8pt);
		\end{scope};
		\end{scope};
		
		\begin{scope}[xshift=-4.05cm]
		\draw [decorate,opacity=0.5,decoration={brace,amplitude=10pt},xshift=-4pt,yshift=0pt]
		(-0.25,0.5) -- (1.8,0.5) node [black,midway,xshift=-0.6cm]{};	
		\end{scope}
		
		\begin{scope}[xshift=-1.85cm]
		\draw [decorate,opacity=0.5,decoration={brace,amplitude=10pt},xshift=-4pt,yshift=0pt]
		(-0.25,0.5) -- (1.8,0.5) node [black,midway,xshift=-0.6cm]{};	
		\end{scope}
		
		\begin{scope}[xshift=0.72cm]
		\draw [decorate,opacity=0.5,decoration={brace,amplitude=10pt},xshift=-4pt,yshift=0pt]
		(-0.25,0.5) -- (1.8,0.5) node [black,midway,xshift=-0.6cm]{};	
		\end{scope}
		
		\begin{scope}[xshift=2.92cm]
		\draw [decorate,opacity=0.5,decoration={brace,amplitude=10pt},xshift=-4pt,yshift=0pt]
		(-0.25,0.5) -- (1.8,0.5) node [black,midway,xshift=-0.6cm]{};	
		\end{scope}
				
		\begin{scope}[xshift=-4.05cm-4pt]
		\draw[opacity=0.5]    (0.775,0.9) to[out=90,in=-100] (1.54,1.5);
		\end{scope}
		
		\begin{scope}[xshift=-1.85cm-4pt]
		\draw[opacity=0.5]    (0.775,0.9) to[out=100,in=-80] (-0.64,1.5);
		\end{scope}
		
		\begin{scope}[xshift=0.72cm-4pt]
		\draw[opacity=0.5]    (0.775,0.9) to[out=90,in=-100] (2.04,1.5);
		\end{scope}
		
		\begin{scope}[xshift=2.92cm-4pt]
		\draw[opacity=0.5]    (0.775,0.9) to[out=100,in=-80] (-0.14,1.5);
		\end{scope}
		
		\begin{scope}[xshift=-2.4cm]
		\draw [decorate,opacity=0.5,decoration={brace,amplitude=10pt,mirror},xshift=-4pt,yshift=0pt]
		(-1.125,1.9) -- (0.925,1.9) node [black,midway,xshift=-0.6cm]{};
		\begin{scope}[xshift=-4pt]
		\draw[opacity=0.5, dotted]    (-1.25,2) -- (-1.25,2.7);
		\draw[opacity=0.5, dotted]    (1.05,2) -- (1.05,2.7);
		\end{scope}
		\begin{scope}[xshift = -4pt-0.75cm]
		\draw[opacity=0.5] (0.2,2.3) circle (8pt);	
		\node[opacity=0.5] at (0.65,2.9) {$\varphi_j^{\mathrm{zak}} = 0$};
		\begin{scope}[xshift = 1.3cm]
		\draw[opacity=0.5] (-0.2,2.3) circle (8pt);
		\end{scope}
		\end{scope}
		\end{scope}
		
		\begin{scope}[xshift=2.87cm]
		\draw [decorate,opacity=0.5,decoration={brace,amplitude=10pt,mirror},xshift=-4pt,yshift=0pt]
		(-1.125,1.9) -- (0.925,1.9) node [black,midway,xshift=-0.6cm]{};
		\begin{scope}[xshift=-4pt]
		\draw[opacity=0.5, dotted]    (-1.25,2) -- (-1.25,2.7);
		\draw[opacity=0.5, dotted]    (1.05,2) -- (1.05,2.7);
		\end{scope}
		\begin{scope}[xshift = -4pt-0.75cm]
		\draw[opacity=0.5] (-0.05,2.3) circle (8pt);	
		\node[opacity=0.5] at (0.65,2.9) {$\varphi_j^{\mathrm{zak}} = \pi$};
		\begin{scope}[xshift = 1.3cm]
		\draw[opacity=0.5] (0.05,2.3) circle (8pt);
		\end{scope}
		\end{scope}
		\end{scope}
		\end{tikzpicture}
		\caption[Single]{Two-dimensional cross-section of a finite dimer chain with 13 resonators, heuristically showing how to identify unit cells with different Zak phases on either side of the edge.  Legend:
			\raisebox{-2pt}{\tikz{
				\draw[pattern=crosshatch dots, opacity=0.3] (0,0) circle (5pt); 
				\draw (0,0) circle (5pt);}} bulk modulus $\kappa$,
			\raisebox{-2pt}{\tikz{
					\draw[pattern=north east lines, opacity=0.3] (0,0) circle (5pt); 
					\draw (0,0) circle (5pt);}} bulk modulus $\overline{\kappa}$, 
			\raisebox{-2pt}{\tikz{\draw (0,0) circle (5pt);}} bulk modulus $\kappa_0$.} \label{fig:finite}
	\end{figure}
	
	\subsection{Numerical illustrations}
	The resonant frequencies and eigenmodes of the finite array with a geometrical defect, composed of 49 resonators, was numerically computed using the multipole method as described in \cite{ammari2019topologically}. \Cref{fig:modeg} shows the localized mode, roughly demonstrating a similar degree of localization as in \Cref{fig:mode}.

	\begin{figure}[h]
		\centering
		\includegraphics[width=0.5\linewidth]{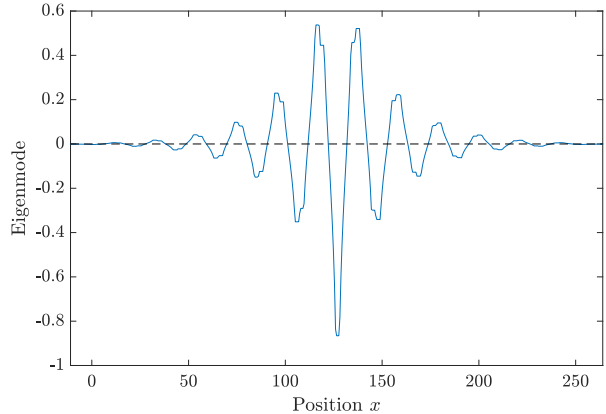}		
		\caption{Plot of the localized mode in the structure with a geometrical defect, in a finite but large array of resonators. Here, we use $\kappa_1 = 1-0.5\iu$ and $\kappa_2 = 1+0.5\iu$.} \label{fig:modeg}
	\end{figure}
	
	\section{Conclusions}
	In this work, we have demonstrated the existence of edge modes in an active system of subwavelength resonators with a defect only in the material parameters. We have linked the continuously varying Zak phase with partial band inversion, which provides a bulk-boundary correspondence for the non-Hermitian system without chiral symmetry. Moreover, we have explicitly computed the  frequency of localized edge modes, and, in accordance to the bulk-boundary correspondence, proved that no edge modes exist in the case of a micro-structure with unbroken $\P\T$-symmetry. 
	
	\bibliographystyle{abbrv}
	\bibliography{pt_topological_3}{}
	\end{document}